\documentclass[11pt,hyp,]{nyjm}
\usepackage{hyperref}
\hypersetup{nesting=true,debug=true,naturalnames=true}
\usepackage{graphicx,amssymb,upref}

\let\<\langle
\let\>\rangle

\let\uml\"

\usepackage{longtable}

\title[Geodesic surfaces in knot complements]{Geodesic surfaces in the complement of knots with small crossing number}

\author{Khanh Le}  
\address[Khanh Le]{Department of Mathematics, Rice University, 6100 Main Street, Houston, TX, 77005} 
\email{khanh.le@rice.edu}
\author{Rebekah Palmer}  
\address[Rebekah Palmer]{Department of Mathematics, Temple University, 1805 N Broad St, Philadelphia, PA 19122} 
\email{rebekah.palmer@temple.edu}
\thanks{The authors were supported by NSF Grant DMS 1906088. The authors would like to thank Nathan Dunfield for his suggestions concerning the computation in the case of the knot $9_{25}$. The authors would like to thank Matthew Stover, Colin Adams, and Nathan Dunfield for their comments on an early draft of the paper. Finally, the authors would like to thank the anonymous referee for pointing out a gap in the proof of Proposition \ref{prop:RootsOfPsi} in an earlier draft and for their comments in improving this paper.}



\usepackage{graphicx} 
\usepackage{subfig} 
\usepackage{float} 
\usepackage{tikz} 

\DeclareMathOperator{\eval}{ev} 
\DeclareMathOperator{\PSL}{PSL} 
\DeclareMathOperator{\PSU}{PSU} 
\DeclareMathOperator{\SL}{SL} 
\DeclareMathOperator{\Stab}{Stab} 
\DeclareMathOperator{\tr}{tr} 
\DeclareMathOperator\rpart{\mathfrak{Re}} 

\newcommand{\bbZ}{\mathbb{Z}} 
\newcommand{\bbQ}{\mathbb{Q}} 
\newcommand{\bbR}{\mathbb{R}} 
\newcommand{\bbC}{\mathbb{C}} 
\newcommand{\bbH}{\mathbb{H}} 
\newcommand{\bbP}{\mathbb{P}} 

\newtheorem{thm}{Theorem}[section]
\newtheorem{lem}[thm]{Lemma}
\newtheorem{cor}[thm]{Corollary}
\newtheorem{prop}[thm]{Proposition}
\newtheorem{quest}[thm]{Question}
\newtheorem{example}[thm]{Example}
\newtheorem{remark}[thm]{Remark}

\begin{document}

\begin{abstract}
In this article, we investigate the problem of counting totally geodesic surfaces in the complement of hyperbolic knots with at most 9 crossings. Adapting previous counting techniques of boundary slope and intersection, we establish uniqueness of a totally geodesic surface for the knots $7_4$ and $9_{35}$. Extending an obstruction to the existence of totally geodesic surfaces due to Calegari, we show that there is no totally geodesic surface in the complement of 47 knots.  
\end{abstract} 
\maketitle
\setcounter{tocdepth}{1}
\tableofcontents


\section{Introduction}

The study of surfaces has been essential in studying the geometry and topology of the 3-manifolds that contain them. In this paper, we will mainly be concerned with complete properly immersed totally geodesic surfaces in hyperbolic 3-manifolds. These surfaces are natural geometric objects which enjoy many applications to the study of the geometry, topology, and algebra of hyperbolic 3-manifolds. For example, Millson constructed families of arithmetic hyperbolic $n$-manifolds for each $n\geq 3$ with arbitrarily large first Betti number by showing the existence of a non-separating totally geodesic hypersurface in these examples \cite{Mill76}. In another application, Adams used geodesic surfaces and cut-and-paste techniques to produce examples of non-homeomorphic hyperbolic 3-manifolds with the same volume \cite[Corollary 4.4]{A85}. Most recently, Bader, Fisher, Miller, and Stover gave a geometric characterization of arithmeticity using geodesic submanifolds: they showed that if a complete finite-volume hyperbolic $n$-manifold of dimension at least 3 contains infinitely many maximal totally geodesic submanifolds then it must be arithmetic \cite[Theorem 1.1]{BFMS21}.  A similar result was also obtained for the case of closed hyperbolic 3-manifolds by Margulis and Mohammadi \cite[Theorem 1.1]{MM22}. 

The converse of the geometric characterization of arithmeticity implies that if a hyperbolic 3-manifold is non-arithmetic, then it contains finitely many (possibly zero) totally geodesic surfaces. We are interested in counting totally geodesic surfaces in non-arithmetic hyperbolic 3-manifolds. Prior to the result in \cite{BFMS21}, there are some known obstructions to the existence of totally geodesic surfaces given by Calegari \cite[Corollary 4.6]{C06} and by Maclachlan and Reid \cite[Theorem 5.3.1 and Corollary 5.3.2]{MR03}. The first examples of non-arithmetic hyperbolic $3$-manifolds in which the set of totally geodesic surfaces is nonempty and finite were given in \cite[Theorem 1.3]{FLMS21} which are link complements in $S^3$ (see \cite[Section 6]{FLMS21} for a concrete description of these examples). However, the method in \cite{FLMS21} did not give the exact count of totally geodesic surfaces. Using homogeneous dynamics, Lindenstrauss and Mohammadi gave an upper bound, with some unknown universal constants, to the number of totally geodesic surfaces in a hyperbolic 3-manifold coming from the Gromov--Piatetski-Shapiro hybrid construction \cite[Theorem 1.4]{LM21}

In \cite{LePalmer}, the authors gave the first explicit examples of infinitely many non-commensurable hyperbolic 3-manifolds each of which contains exactly $k$ totally geodesic surfaces for every positive integer $k$ \cite[Theorem 1.2]{LePalmer}. In particular, the authors showed that the complement of an infinite family of twist knots contains a unique totally geodesic surface and considered finite covers of these twist knot complements to produce hyperbolic 3-manifolds with the desired number of totally geodesic surfaces \cite[Theorem 1.3]{LePalmer}. To establish uniqueness of the totally geodesic surface in a family of twist knot complements, the authors introduced counting techniques which take advantage of the geometry and number theoretic properties of these knot complements. In this work, we study totally geodesic surfaces in the complement of hyperbolic knots with at most nine crossings. The goal of this paper is to extend the current counting techniques of totally geodesic surfaces and known obstructions to the existence of these surfaces in order to gain a quantitative understanding of totally geodesic surfaces using a finite collection of hyperbolic knots as a testing ground. This quantitative understanding sheds light on the limitations of the current techniques and points us to some interesting questions which are discussed in Section \ref{sec:Questions}.

We first summarize what was known about totally geodesic surfaces in knots with at most nine crossings. Among the knots with at most nine crossings, there are 79 hyperbolic knots \cite{knotinfo}. Let 
\[
\mathcal{K} = \{ \text{Hyperbolic knots in $S^3$ with at most nine crossings}\}.
\]
The set $\mathcal{K}$ includes all prime knots of nine or fewer crossings other than the unknot, $3_1$, $5_1$, $7_1$, $8_{19}$ and $9_1$. The knot $4_1$, also known as the figure-8 knot, produces the only arithmetic knot complement in $S^3$ \cite[Theorem 2]{Reid91b}. The complement $S^3 \setminus 4_1$ contains an immersed thrice-punctured sphere which must be totally geodesic \cite[Theorem 3.1]{A85}. It follows that the complement of the figure-8 knot contains infinitely many totally geodesic surfaces. The complements of the remaining hyperbolic knots contain finitely many totally geodesic surfaces because they are all non-arithmetic. It was observed that the complement of each of the following knots 
\[
\{
8_{10},8_{15},9_{22},9_{32},9_{35},9_{42},9_{48}
\}
\]
does not contain any closed totally geodesic surfaces \cite[Section 4.3]{Reid91b}. Menasco and Reid observed the complement of an alternating knot, a closed 3-braid or a tunnel number one knot does not contain any closed embedded totally geodesic surface \cite[Theorem 1 and Corollary 4]{MeR92}. Adams and Schoenfeld proved that the complements of two-bridge knots do not contain any embedded orientable totally geodesic surface \cite[Theorem 4.1]{AS05}. Using the data in \cite{knotinfo}, we see the complement of the knots in 
\[
\mathcal{K} \setminus \{9_{46},9_{47},9_{48},9_{49}\}
\]
does not contain any closed embedded totally geodesic surface.

The complements of the knots $5_2$, $6_1$, $7_2$, $8_1$, and $9_2$ each contain an immersed totally geodesic thrice-punctured sphere. This surface is known to be unique in the case of the knot $5_2$, $7_2$, and $9_2$ \cite[Corollary 1.4]{LePalmer}.  In this article, we show that the complement of the knot $7_4$ contains an immersed totally geodesic twice-punctured torus. To the best of our knowledge, this surface has not been found previously.

\begin{thm} \label{t:7-4!surf}
The complement of the knot $7_4$ contains a unique totally geodesic surface.  Moreover, this surface is a twice-punctured torus.
\end{thm}

The complement of the knot $9_{35}$ contains a totally geodesic Seifert surface which was found by Adams and Schoenfeld. More generally, Adams and Schoenfeld observed that the complements of the balanced pretzel knots each contain a totally geodesic Seifert surface which is unique among Seifert surfaces (see \cite[Example 3.1]{AS05} and \cite[Corollary 3.4]{Ad.et.al08}). See Figure \ref{fig:9-35_seifert} for a diagram of the 3-tangle balanced pretzel knot $P(n,n,n)$, where $n$ is the number of half twists in each tangle, with its totally geodesic Seifert surface. The knot $9_{35}$ is the knot with the smallest crossing number in this family, also denoted as $P(3,3,3)$.  We herein study the infinite family of 3-tangle balanced pretzel knots $P(n,n,n)$ (see Figure \ref{fig:pretzel_seifert} for the knot diagram). For the complement of the knots in this family, we show that:

\begin{thm} \label{t:pret!surf}
Suppose $n$ is an odd prime. Any totally geodesic surface $\Sigma$ in the complement of $P(n,n,n)$ must:
\begin{itemize}
    \item be the totally geodesic Seifert surface $S$, or
    \item intersect $S$ transversely along a union of closed geodesics.
\end{itemize}
Furthermore, $S$ is the unique totally geodesic surface in the complement of $P(3,3,3)$.
\end{thm}

\subsection{Proof outline of 
Theorem \ref{t:7-4!surf} and Theorem \ref{t:pret!surf}} The proofs of Theorem \ref{t:7-4!surf} and Theorem \ref{t:pret!surf} follow the same outline as that of \cite[Theorem 1.3]{LePalmer}. In particular, we take advantage of the facts that:
\begin{itemize}
    \item the traces of the knot group, identified as a discrete subgroup of $\PSL_2(\bbC)$, are algebraic integers;
    \item the trace fields of the knot group have odd degree and contain no real subfield besides $\bbQ$.
\end{itemize}
These conditions imply that the complements of these knots do not contain any closed totally geodesic surfaces by a proposition of Reid (\cite[Proposition 2]{Reid91}; restated in this article as Proposition \ref{prop:NoClosedTG}). We observe that the number theoretic constraints above also imply that the traces of any Fuchsian subgroup of these knot groups must be integers, which we will refer to as the trace condition (see Section \ref{subsec:BoundarySlopeTraceCondition} and \eqref{eq:TraceCondition}). Taking advantage of the trace condition and the geometry of the cusp neighborhood, we show that any totally geodesic surface in these knot complements must either be the known totally geodesic surface or intersect the torus neighborhood of the cusp in parallel with the known totally geodesic surface.  The definition of parallel here is in reference to boundary slope; see Lemma \ref{lem:7-4_bdry} and Lemma \ref{lem:pret0} for a more precise description.  We also observe a geometric constraint that totally geodesic surfaces must intersect each other along a union of closed and cusp-to-cusp geodesics (see Lemma \ref{lem:IntersectionsOfTGS}). By playing the trace condition, boundary slope, and geometric constraints off each other, we prove that the only totally geodesic surface satisfying all three constraints is the known totally geodesic surface, and therefore we establish uniqueness of totally geodesic surface in these examples. The new idea compared to the techniques in \cite{LePalmer} is the use of closed geodesics in the final step to establish uniqueness. Furthermore, we also demonstrate that the techniques in \cite{LePalmer} can also be adapted to the situation where the set of permissible boundary slopes of totally geodesic surfaces is a singleton, as in the case of the pretzel knot $9_{35}$.     

In the process of proving Theorem \ref{t:pret!surf}, we obtain new information about the trace field of the family of 3-tangle balanced pretzel knots $P(2k+1,2k+1,2k+1)$ for $k\geq 1$, which is of independent interest. In particular, we prove the following.

\begin{cor}
\label{cor:PretzelTraceField}
The trace field of $P(2k+1,2k+1,2k+1)$ is $\bbQ(z_k)$ where the minimal polynomial of $z_k$ is 
of degree $2k+1$. 
\end{cor}

To our best knowledge, the only other infinite family of knots whose trace field is precisely known is the family of twist knots \cite[Theorem 1]{HS01}. To put Corollary \ref{cor:PretzelTraceField} in a broader context, we observe that both families of knots can be obtained by doing Dehn surgery on a hyperbolic link complement. The twist knots are obtained by doing Dehn surgery on the Whitehead link. Meanwhile, the balanced pretzel knots are obtained by doing Dehn surgery on a fully augmented pretzel link (see \cite{MMT20} for the definition). In other words, Corollary \ref{cor:PretzelTraceField} gives another example of an infinite family of Dehn surgeries on a link whose trace field degree has precise linear growth. Corollary \ref{cor:PretzelTraceField} gives an illustrative example of the conditional theorem in \cite[Theorem 1.2]{GB21}. 

\subsection{Obstruction results}
Now we turn our attention to the remaining $73$ hyperbolic knots:
\[
\mathcal{K} \setminus \{4_1,5_2,7_2,7_4,9_2,9_{35}\}
\]
Using SnapPy \cite{SnapPy} and \cite{knotinfo}, we found $23$ fibered knots that satisfy the obstruction to the existence of totally geodesic surface due to Calegari \cite[Corollary 4.6]{C06}. We modify this obstruction for the non-fibered case to obtain the following.

\begin{thm}
\label{thm:ModifiedCalegari}
Let $M$ be a hyperbolic knot complement such that the trace field $K$ of $M$ has no proper real subfield beside $\bbQ$ and contains no quadratic field. Let $F$ be a Seifert surface in $M$ with minimal genus $g$ among Seifert surfaces. Suppose that there exists $\rho:\pi_1(M) \to \PSL_2(\bbR)$ a Galois conjugate of the geometric representation of $\pi_1(M)$ such that 
\[
e_\rho([F]) < 2g - 1
\]
where $e_\rho\in H^2(M,\partial M;\bbZ)$ is the relative Euler class of this representation.
 Then $M$ contains no totally geodesic surfaces.
\end{thm}

We found an additional $24$ hyperbolic knots to which Theorem \ref{thm:ModifiedCalegari} can be applied and thus which have no totally geodesic surfaces in their complement. Together, we have the following result.

\begin{thm}
\label{thm:TGSKnotsUnderNineCrossings}
Among the $79$ hyperbolic knots in $\mathcal{K}$, the complements of the following 47 knots 
\begin{equation}
  \begin{aligned}
        \{ 
        & 6_2,7_3, 7_5, 7_6, 8_2, 8_4, 8_5, 8_6, 8_7, 8_{10}, 8_{14},  8_{15}, 8_{16}, 8_{20}, 9_3, \\
        & 9_4, 9_6, 9_7, 9_8, 9_9, 9_{10}, 9_{11}, 9_{12}, 9_{13}, 9_{15}, 9_{16}, 9_{17}, 9_{18}, 9_{20}, 9_{21}, \\
        & 9_{22}, 9_{23}, 9_{24}, 9_{25}, 9_{26}, 9_{29}, 9_{31}, 9_{32}, 9_{34}, 9_{36}, 9_{38}, 9_{39}, 9_{42}, 9_{43}, 9_{45}, 9_{48}, 9_{49}
        \}
    \end{aligned}    
\label{eq:NoTGSKnots}
\end{equation}
contain no totally geodesic surfaces.
\end{thm}

\subsection{Summary}
We summarize known results about totally geodesic surfaces in the complement of knots in $\mathcal{K}$ as follows:
\begin{itemize}
    \item the complement of the knot $4_1$ contains infinitely many totally geodesic surfaces \cite[Corollary 1]{Reid91b}.
    \item the complement of the knots $5_2$, $7_2$, $7_4$, $9_2$ and $9_{35}$ contains a unique totally geodesic surface.
    \item the complement of the knots in the set
    \begin{align*}
        \{ 
        & 6_2,7_3, 7_5, 7_6, 8_2, 8_4, 8_5, 8_6, 8_7, 8_{10}, 8_{14},  8_{15}, 8_{16}, 8_{20}, 9_3, \\
        & 9_4, 9_6, 9_7, 9_8, 9_9, 9_{10}, 9_{11}, 9_{12}, 9_{13}, 9_{15}, 9_{16}, 9_{17}, 9_{18}, 9_{20}, 9_{21}, \\
        & 9_{22}, 9_{23}, 9_{24}, 9_{25}, 9_{26}, 9_{29}, 9_{31}, 9_{32}, 9_{34}, 9_{36}, 9_{38}, 9_{39}, 9_{42}, 9_{43}, 9_{45}, 9_{48}, 9_{49}
        \}
    \end{align*}
    contains no totally geodesic surfaces.
        \item the complement of the knots in the set
    \[
\mathcal{K} \setminus \{9_{46},9_{47},9_{48},9_{49}\}
\]
does not contain any closed embedded totally geodesic surface \cite[Theorem 1 and Corollary 4]{MeR92}.
\end{itemize}
This is by no means a comprehensive description of totally geodesic surfaces in $\mathcal{K}$. We suspect that the knot $9_{41}$ contains an immersed totally geodesic surface with cusps in its complement. We also think that there are no totally geodesic surfaces in the remaining knot complements. See Section \ref{sec:Questions} for further discussion.    

\subsection{Outline of the paper} In Section \ref{sec:prelim}, we recall some results that are used in the proof of Theorem \ref{t:7-4!surf} and Theorem \ref{t:pret!surf}. We will also recall Calegari's obstruction to the existence of totally geodesic surfaces and give a prove of Theorem \ref{thm:ModifiedCalegari}. In Section \ref{sec:74}, we give a proof of Theorem \ref{t:7-4!surf}. In Section \ref{sec:BalancedPretzelKnots}, we give a description of the geometric representation and the trace field of the family of 3-tangle pretzel knots in Theorem \ref{thm:PretzelTraceField} and Corollary \ref{cor:PretzelTraceField}. Using these results, we prove Theorem \ref{t:pret!surf}. In Section \ref{sec:SeifertSurfacesEulerClassObstructions}, we outline our computational approach in proving Theorem \ref{thm:TGSKnotsUnderNineCrossings}. In Section \ref{sec:Questions}, we discuss some interesting questions arising from this paper. 



\section{Preliminaries}
\label{sec:prelim}
We begin with presenting some previously known statements about the behavior of totally geodesic surfaces in hyperbolic $3$-manifolds. 

\subsection{Geometric and arithmetic constraints on totally geodesic surfaces}

The following proposition of Reid \cite[Proposition 2]{Reid91} gives an arithmetic constraint on the existence of closed totally geodesic surfaces.  In particular, it rules out the existence of closed totally geodesic surfaces.
\begin{prop}
\label{prop:NoClosedTG}
Let $\Gamma$ be a non-cocompact Kleinian group of finite covolume and satisfying the following two conditions:
\begin{itemize}
	\item $\bbQ(\tr \Gamma)$ is of odd degree over $\bbQ$ and contains no proper real subfield other than $\bbQ$.  
	\item The traces of $\Gamma$ are algebraic integers.
\end{itemize}
Then $\Gamma$ contains no cocompact Fuchsian groups and at most one commensurability class (up to conjugacy in $\PSL_2(\bbC)$) of non-cocompact Fuchsian subgroup of finite covolume.
\end{prop}	

For a finite covolume Kleinian group $\Gamma$, we say that $p\in \partial_\infty \bbH^3$ is a \emph{cusp point} if $p$ is a fixed point of a parabolic isometry in $\Gamma$. We say a geodesic in a cusped finite-volume hyperbolic $3$-manifold $\bbH^3 /\Gamma$ is a \emph{cusp-to-cusp geodesic} if it is the image of a geodesic in $\bbH^3$ connecting two cusp points under the action of $\Gamma$ on $\bbH^3$.  The following lemma of Fisher, Lafont, Miller, and Stover \cite[Lemma 3.1]{FLMS21} describes the intersection of totally geodesic hypersurface and immersed totally geodesic submanifolds in finite-volume hyperbolic $n$-manifold. We restate their lemma for $n=3$.

\begin{lem}
\label{lem:IntersectionsOfTGS}
Let $M$ be a complete finite volume hyperbolic $3$-manifold with at least $1$ cusp. Suppose that $\Sigma_1$ and $\Sigma_2$ are two distinct properly immersed totally geodesic surfaces in $M$ such that $\Sigma_1 \cap \Sigma_2$ is nonempty. Then $\Sigma_1 \cap \Sigma_2$ is the union of closed geodesics and cusp-to-cusp geodesics.
\end{lem}

\subsection{Boundary slope and the trace condition}
\label{subsec:BoundarySlopeTraceCondition}
A primary conceptual tool in proving Theorem \ref{t:7-4!surf} and Theorem \ref{t:pret!surf} is using the boundary slope of a totally geodesic surface in concert with Lemma \ref{lem:IntersectionsOfTGS}. There are two approaches to boundary slopes --- one in the manifold itself and one in the universal cover.

Let $M$ be the complement of a hyperbolic knot $J$ in $S^3$ and $T$ be the torus boundary of a small tubular neighborhood of $J$ in $S^3$. When $J$ is not the unknot, the fundamental group of $T$ injects into the fundamental group of $M$. We fix a basis for $\pi_1(T) \cong \bbZ^2$ by choosing $a,\ell \in \pi_1(M)$ such that $a$ is a meridian with corresponding homological longitude $\ell$ of the knot $J$. In a hyperbolic knot complement $M$, the neighborhood of each cusp of a cusped totally geodesic surface $\Sigma$ must intersect the torus neighborhood $T$ of the knot. Because $\Sigma$ has finite area, each intersection is a closed curve on the embedded torus neighborhood $T$; thus, each curve represents the element $a^p\ell^q \in \pi_1(M)$ up to conjugation. The ratio $p/q \in \bbQ \cup \{\infty\}$ is the \emph{boundary slope} of the corresponding cusp of $\Sigma$.

This topological description can be expressed equivalently in the universal cover. Let us identify the universal cover $\mathbb{H}^3$ of $M$ with the upper half-space model and the \emph{visual boundary} $\partial_\infty \bbH^3$ with $\mathbb{C} \cup \{\infty\}$. The action of $\pi_1(M)$ on $\bbH^3 \cup \partial_\infty \bbH^3$ is given by a discrete faithful representation. Up to conjugation, we assume that under the discrete faithful representation $\rho:\pi_1(M) \to \PSL_2(\bbC)$, the images $\rho(a)$ and $\rho(\ell)$ are parabolic isometries fixing $\infty$ as a cusp point. Let $\Sigma$ be a properly immersed cusped totally geodesic surface in $M$. Since $M$ has one cusp and $\Sigma$ is properly immersed, the neighborhood of each cusp of $\Sigma$ must be contained in the cusp neighborhood of $M$. Given any cusp of $\Sigma$, we can consider a lift $\widetilde{\Sigma}$ of $\Sigma$ to $\bbH^3$ by putting the cusp point of $\Sigma$ at $\infty$. The hyperplane $\widetilde{\Sigma}$ intersects a horoball based at $\infty$ along some horocycle. The stabilizer of this horocycle is generated by an isometry of the form $\rho(a^p\ell^q)$.  See Figure \ref{fig:bdry_slope} for a visualization of a totally geodesic surface with a boundary slope $1/0$ and the view from $\infty$ of its lifts to the universal cover.

\begin{figure}
  \captionsetup{width=0.3\linewidth}
    \centering
    \subfloat[\centering Cusp in knot complement with knot neighborhood]{\hspace{8mm}\includegraphics[height=5cm]{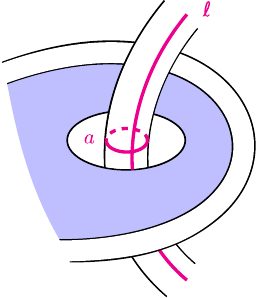}}
    \hspace{2cm}
    \subfloat[\centering View from $\infty$ in universal cover]{{\hspace{8mm}\includegraphics[height=5cm]{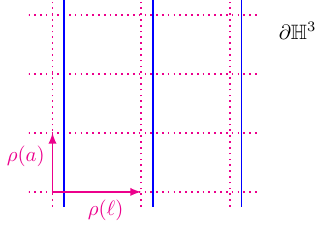}}}
  \captionsetup{width=\linewidth}
    \caption{Totally geodesic surface (blue) with boundary slope $1/0$}
    \label{fig:bdry_slope}
\end{figure}

Under some mild condition on the traces of $\rho(\pi_1(M))$, the complete set of boundary slopes of all totally geodesic surfaces in $M$ are computed in Lemma \ref{lem:7-4_bdry} and Lemma \ref{lem:pret0} (in the style of \cite[Lemma 3.5]{LePalmer}). The following is a critical preliminary calculation in our analysis. Up to a further conjugation, we assume that 
\[\rho(a) = \begin{pmatrix}
1 & 1 \\ 0 & 1
\end{pmatrix} \quad\text{and}\quad \rho(\ell) = \begin{pmatrix}
-1 & -\tau \\ 0 & -1 
\end{pmatrix}.\]For convenience, let us momentarily drop $\rho$ from our notation and identify elements of $\pi_1(M)$ with its image in $\PSL_2(\bbC)$ under $\rho$. We suppose that the trace field $\bbQ(\tr \pi_1(M))$ contains no proper real subfield besides $\bbQ$ and the elements of $\tr \pi_1(M)$ are all algebraic integers. Let $\Sigma$ be a totally geodesic surface in $M$. We note that the set $\tr \pi_1(\Sigma)$ must contain only real algebraic integers in $\bbQ(\tr \pi_1(M))$. The condition that $\bbQ(\tr \pi_1(M))$ does not contain any proper real subfield other than $\bbQ$ implies that $\tr \pi_1 (\Sigma) \subseteq \bbZ$. Suppose further that $\Sigma$ has cusps with boundary slopes $p/q$ and $m/n$. Then there exist hyperplane lifts $\widetilde{\Sigma}_1$ and $\widetilde{\Sigma}_2$ whose visual boundaries contain $\infty$ and that are stabilized by $a^p \ell^q$ and $a^m \ell^n$, respectively. There exists 
\begin{equation*}
\gamma = \begin{pmatrix}
\alpha & \beta \\ \delta & \eta
\end{pmatrix} \in \pi_1(M)
\end{equation*}
such that $\gamma(\widetilde{\Sigma}_2) = \widetilde{\Sigma}_1$. The conjugate $\gamma a^m \ell^n \gamma^{-1}$ is in $\Stab(\widetilde{\Sigma}_1)$ because $a^m \ell^n \in \Stab(\widetilde{\Sigma}_2)$.  Since $a^p \ell^q \in \Stab(\widetilde{\Sigma}_1)$ as well, the trace of their product
\begin{equation*}
  \tr(\gamma a^m \ell^n \gamma^{-1} a^p \ell^q) 
  = (-1)^{n+q+1}[-2 + (m+n \tau)(p+q \tau) \delta^2 ]
\end{equation*}
must be in $\bbZ$, which is true if and only if
\begin{equation}
  \label{eq:TraceCondition}
  \delta^2(nq\tau^2 + (mq + np)\tau + mp ) \in \bbZ
\end{equation}
We shall refer to \eqref{eq:TraceCondition} as the trace condition. When $M$ has at least one cusp, the entries of $\pi_1(M)$ can be taken to be in the trace field. By rewriting \eqref{eq:TraceCondition} in terms of a $\bbQ$-basis of the trace field, we obtain a set of equations that the boundary slopes of $\Sigma$ must satisfy. Solving this set of equations allows us to compute the complete set of boundary slopes of totally geodesic surfaces in $M$.  

\subsection{Trace field and orientability of totally geodesic surfaces}

\begin{lem}
\label{lem:nonori}
If an orientable hyperbolic 3-manifold $M = \bbH^3/\Gamma$ contains a non-orientable totally geodesic surface, then the trace field of $\Gamma$ contains either a real subfield properly containing $\bbQ$ or contains an imaginary quadratic field.
\end{lem}

\begin{proof}
If an orientable hyperbolic $3$-manifold $M$ contains a non-orientable totally geodesic surface, then there must exist $\gamma \in\Gamma$ such that the image of $\gamma$ under the discrete faithful representation is conjugate into 
\[
\begin{pmatrix}
i & 0 \\ 0 & -i
\end{pmatrix}\PSL_2(\bbR).
\] 
The trace field of the $3$-manifold must contain purely imaginary elements, say $\alpha \in i\bbR$. If $\alpha^2 \in \bbQ$, then $\bbQ(\alpha)$ is an imaginary quadratic subfield of the trace field of $\Gamma$. Otherwise, the trace field of $\Gamma$ contains a real subfield $\bbQ(\alpha^2)$ that properly contains $\bbQ$.
\end{proof}

\begin{remark} 
\label{rem:nonori}
If an orientable hyperbolic $3$-manifold $M$ contains a non-orientable totally geodesic surface, then there must exist $\gamma \in \pi_1(M)$ such that the image of $\gamma$ under the discrete faithful representation is conjugate to the product of an element of $\PSL_2(\bbR)$ and the order $2$ rotation diagonal matrix.
Then the trace field of the $3$-manifold must contain purely imaginary elements.
This is impossible when the trace field of odd degree.  Since almost all knot complements we will address in this article will have odd trace field, we will safely presume the exclusion of non-orientable surfaces.
We will individually address relevant knot complements whose trace field has even degree.
\end{remark}

\subsection{Calegari's obstruction of totally geodesic surfaces using Euler class}

Finally, we will recall a method introduced by Calegari in \cite{C06} to obstruct the existence of totally geodesic surfaces in certain fibered knot complement in rational homology sphere \cite[Corollary 4.6]{C06}. We will start by recalling the definition of the Euler class associated to a $\PSL_2(\mathbb{R})$-representation of $\pi_1(M)$ and the definition of Thurston norm. 

\subsubsection{The Euler class and Thurston norm}
\label{subsec:EulerClass}

Let $M$ be the complement of a knot in a rational homology 3-sphere and $\rho: \pi_1(M) \to \PSL_2(\bbR)$ be a representation such that $\rho(\pi_1(\partial M))$ is parabolic. Since $\PSL_2(\bbR)$ acts on $\mathbb{RP}^1$, we have the associated circle bundle of $M$ defined by  
\begin{equation}
    \label{eq:Erho}
    E_\rho = \Tilde{M} \times \mathbb{RP}^1/ (x,p) \sim (\gamma\cdot x,\rho(\gamma)(p)).
\end{equation}
The obstruction of finding a section of $E_\rho$ is measured by the Euler class $e_\rho \in H^2(M;\bbZ)$. When $M$ is the complement of a knot in a rational homology 3-sphere, $H^2(M;\bbZ) = 0$. Therefore, the Euler class invariant vanishes for $M$. Nevertheless, we can still define a relative Euler class when $\rho(\pi_1(\partial M))$ is parabolic. Since $\pi_1(\partial M)$ is abelian, the image $\rho(\pi_1(\partial M))$ has a unique fixed point. This fixed point defines a canonical section of $E_\rho|_{\partial M}$ over $\partial M$. The obstruction of extending this section over $M$ is measured by the relative Euler class, which we also denote as $e_\rho \in H^2(M,\partial M;\bbZ)$.  

We can describe the relative Euler class $e_\rho$ in terms of the homomorphism $\rho$ as follows. Since $M$ is the complement of a knot in a rational homology 3-sphere, $H_2(M,\partial M;\bbZ)$ is generated by $[F]$ where $F$ is a Seifert surface of $M$. The class $e_\rho$ is completely determined by $e_\rho([F])$. The group $H^2(M;\bbZ)$ vanishes, so we get a lift $\widetilde{\rho}:\pi_1(M) \to \widetilde{\PSL}_2(\bbR)$ of $\rho$. This lift determines an image $\widetilde{\rho}(\partial F)$ in $\widetilde{\PSL}_2(\bbR)$. The following lemma appeared in \cite[Section 2.5]{G88}.

\begin{lem}
\label{lem:LiftingLongitudeIsWellDefined}
The element $\widetilde{\rho}(\partial F)$ is independent of the choices of lifts of $\rho$ and the choices of Seifert surface representing the generator of $H_2(M,\partial M;\bbZ)$.
\end{lem}

\begin{proof}
Since $F$ is a Seifert surface, the boundary $\partial F$ is identified with a well-defined element of $\pi_1(\partial M)$ which is the generator of $\ker (\pi_1(\partial M) \to H_1(M;\bbZ))$. Lifts of $\rho$ are parametrized by $H^1(M;\bbZ)$. In particular given a lift $\widetilde{\rho}$ of $\rho$ and a element $\phi \in H^1(M;\bbZ)$, we obtain a different lift $\widetilde{\rho}^{\phi}:\pi(M) \to \widetilde{\PSL}_2(\bbR)$ by
\[
\widetilde{\rho}^{\phi}(\gamma) = \widetilde{\rho}(\gamma) c^{\phi(\gamma)}
\]
where $c$ is a generator of the center of $\widetilde{\PSL}_2(\bbR)$. The element $\partial F$ is in the commutator subgroup of $\pi_1(M)$, so $\partial F \in \ker (\phi)$ for all $\phi \in H^1(M;\bbZ)$. Therefore, the image of $\partial F$ is independent of the choice of lift of $\rho$.  
\end{proof}

The canonical section of $E_\rho|_{\partial M}$  determines a section of $\pi_1(\partial M)$ 
\begin{equation}
\label{eq:CanSection}
    s: \pi_1(\partial M) \to \widetilde{\PSL}_2(\bbR)
\end{equation}
as follows. Let us identify $\bbR \bbP^1 \cong \bbR / \pi \bbZ$. The image $\rho(\pi_1(\partial M))$ is parabolic and fixes a unique point $p \in \mathbb{RP}^1$ which has a unique lift $\widetilde{p}$ in the interval $[0,\pi)$. The canonical section $s:\pi_1(\partial M)\to \widetilde{\PSL}_2\bbR$ is obtained by lifting $\rho(\pi_1(\partial M))$ to elements fixing $\widetilde{p}$. Since $s(\ell)$ and $\widetilde{\rho}(\ell)$ have the same image in $\PSL_2(\bbR)$, it follows that $\widetilde{\rho}(\ell)=s(\ell)c^n$ for some $n \in\bbZ$. The integer $n$ is $e_\rho([F])$. 

Another norm that we have on $H_2(M,\partial M;\bbR)$ is the Thurston norm. For an irreducible and atoroidal manifold $M$ with boundary, Thurston introduced a norm on $H_2(M,\partial M;\mathbb{R})$ in \cite{T86}. Given a homology class $[S]$ in $H_2(M,\partial M;\bbZ)$, the Thurston norm of $[S]$ is defined to be
\[
||[S]|| = \inf\{ -\chi(F) \mid F \text{ represents [S]} \}
\]
where $F$ contains no sphere components. The function $||\cdot ||$ is extended to each ray containing an integral point by linearity. Finally, $||\cdot||$ is extended continuously to $H_2(M,\partial M;\mathbb{R})$ by convexity. 

\subsubsection{Obstructing totally geodesic surfaces using Euler class}

Using the Euler class and Thurston norm on $H_2(M,\partial M;\bbZ)$, Calegari produced an obstruction to the existence of totally geodesic surfaces in fibered knot complements in a rational homology sphere \cite[Corollary 4.6]{C06}. When $M$ is a fibered knot complement, Calegari showed that for every Galois conjugate of the hyperbolic representation into $\PSL_2(\bbR)$, \[e_\rho([F]) < \Vert[F]\Vert,\] where $F$ is the fiber surface and $\Vert \cdot \Vert$ is the Thurston norm on $H_2(M,\partial M;\bbZ)$ \cite[Remark 3.3]{C06}. Under some assumption on the trace field of the knot $K$, the inequality above rules out the existence of totally geodesic surfaces. In practice, we just need the inequality to hold at one real place.

Inspired by this idea, we modify Calegari's condition to produce an obstruction to the existence of totally geodesic surfaces in non-fibered knot complements given in Theorem \ref{thm:ModifiedCalegari}. Before proving Theorem \ref{thm:ModifiedCalegari}, we need the following:

\begin{thm}[{\cite[Theorem 4.4]{C06}}]
\label{thm:NoSeparatingTGS}
Let $M$ be a cusped hyperbolic 3-manifold, and suppose $S \subset M$ is a totally geodesic surface with rational traces (possibly immersed). If S is not (Gromov or Thurston) norm minimizing
in its homology class, the trace field $K$ has no real places. In particular, if $K$ has a real place then $M$ does not contain any null-homologous totally geodesic surface with rational traces. 
\end{thm}

\begin{proof}[Proof of Theorem \ref{thm:ModifiedCalegari}]
For a contradiction suppose that $S$ is a totally geodesic surface in $M$. Since the trace field $K$ contains no proper real subfield besides $\bbQ$ and contains no quadratic subfield, $S$ is orientable and has rational traces. Since $\rho$ is a Galois conjugate of the geometric representation, $\rho(S)$ remains a Fuchsian subgroup of finite coarea in $\PSL_2(\bbR)$. Consequently, $|e_\rho([S])| = -\chi(S)$. 

Let $F \subset M$ be a Seifert surface of minimal genus. Since $M$ is a knot complement, $H_2(M,\partial M;\bbZ)$ is generated by $[F]$ and $[S] = n[F]$ for some $n\in\bbZ$. Note that since $K$ has a real place, $n\neq 0$ by Theorem \ref{thm:NoSeparatingTGS}. Since $F$ is a minimal genus Seifert surface, it is Thurston norm minimizing, and therefore $-|n|\chi(F)\leq -\chi(S)$. By assumption $|e_\rho([F])| < 2g-1$, we then have
\[
-\chi(S) = |e_\rho([S])| = |n|\cdot| e_\rho([F])| < |n|\cdot(2g-1) = -|n|\chi(F)\leq -\chi(S),
\]
which is the desired contradiction. 
\end{proof}

\begin{remark}
As we shall see in Section \ref{sec:SeifertSurfacesEulerClassObstructions}, the inequality in Theorem \ref{thm:ModifiedCalegari} holds for 
more than half of the hyperbolic knots with fewer than nine crossings. 
\end{remark}


\section{The knot $7_4$}
\label{sec:74}

Throughout this section, we let $J$ be the knot $7_4$ in Figure \ref{fig:7_4_diagram}, $M$ the complement of $J$ in $S^3$, and $\Gamma$ the fundamental group of $M$. 

\begin{figure}[H]
\centering
\includegraphics[width=1in,angle=90]{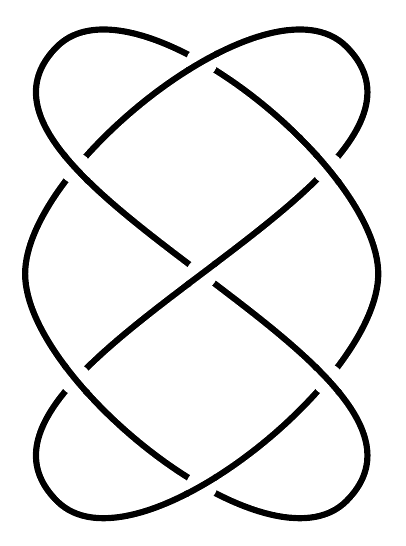}
\caption{The knot $7_4$}
\label{fig:7_4_diagram}
\end{figure}

\subsection{Trace field and totally geodesic surface}
The knot $J$ is a two-bridge knot that corresponds to the fraction $15/11$. Following \cite[Section 4.5]{MR03}, the knot group $\Gamma$ has the following presentation.
\begin{equation}
\label{eq:7_4_GroupPresentation}
\Gamma = \langle x,y \mid xw = wy \rangle,
\end{equation}
where $w = yx^{-1}yxy^{-1}xy^{-1}x^{-1}yx^{-1}yxy^{-1}x$. The homological longitude of $J$ is given by $\ell = wvx^{-4}$ where $v$ is the word $w$ spelled backwards.

The manifold $M$ admits a hyperbolic structure with the discrete and faithful representation $\rho:\Gamma \to \PSL_2(\bbC)$ given by
\begin{equation}
\label{eq:7_4_Holonomy}
    \rho(x) = \begin{pmatrix} 1 & 1 \\ 0 & 1\end{pmatrix} \quad \text{and} \quad \rho(y) = \begin{pmatrix} 1 & 0 \\ z & 1\end{pmatrix}
\end{equation}
where $z$ is a complex root of the polynomial $\Lambda =z^3 - 4z^2 + 4z + 1$. Note that the group relation in \eqref{eq:7_4_GroupPresentation} holds if and only if $z$ satisfies a polynomial of degree 7. This polynomial factors into two factors of degree 3 and 4. A quick check using SnapPy \cite{SnapPy} tells us that the hyperbolic structure corresponds to the complex root of the cubic factor. 

Since the representation $\rho$ is faithful, we can identify $\Gamma$ with its image under $\rho$. In this identification, we can calculate $\ell$ to be
\begin{equation*}
    \ell
    = \begin{pmatrix}
    -1 & 2 (2z^2 - 6z + 5) \\ 0 & -1
    \end{pmatrix}.
\end{equation*}
If we set $\tau = -2 (2z^2 - 6z + 5)$, then $\tau+2 = -4 (z-1)(z+2)$.

We observe that the trace field of $\Gamma$ is $\bbQ(z)$ which is a cubic extension over $\bbQ$. Since $z$ is an algebraic integer, \eqref{eq:7_4_Holonomy} implies that $\Gamma$ has integral traces. By Proposition \ref{prop:NoClosedTG}, $\Gamma$ does not contain cocompact Fuchsian groups and only contains non-cocompact Fuchsian subgroup commensurable to $\PSL_2(\bbZ)$. In fact, we prove that $M$ contains a totally geodesic twice-punctured torus. 

\begin{prop}
\label{prop:7_4_TGS_TwicePuncturedTorus}
The manifold $M$ contains a totally geodesic twice-punctured torus. The fundamental group of this twice-punctured torus is generated by
\begin{equation}
    \begin{array}{r@{~}l}
    \Delta 
    & = \left\langle 
    \begin{pmatrix} -1 & 4(z-1)(z-2) \\ 0 & -1 \end{pmatrix}, 
    \begin{pmatrix} 5 & 3(z-1)(z-2) \\ z^2-z-1 & -1 \end{pmatrix}, \right.
    \\
    & \hspace{2cm} \left.
    \begin{pmatrix} 7 & 11(z-1)(z-2) \\ z^2-z-1 & -3 \end{pmatrix}
    \right\rangle
    \end{array}
\end{equation}
which come from the words
\begin{align*}
a &= \begin{pmatrix} -1 & 4(z-1)(z-2) \\ 0 & -1 \end{pmatrix}
 = x^2 \ell
\\
b &= \begin{pmatrix} 5 & 3(z-1)(z-2) \\ z^2-z-1 & -1 \end{pmatrix}
 = wy^{-1}xy^{-1}xy^{-1}
\\ 
c &= \begin{pmatrix} 7 & 11(z-1)(z-2) \\ z^2-z-1 & -3 \end{pmatrix}
 = x^{-1}wxy^{-1}xw^{-1}x^2w^{-1}x
\end{align*}
The boundary slopes of the two cusps are $\pm 2$.
\end{prop}

\begin{proof}
We first observe that $\Delta \subset \Gamma$ since the generators can be expressed as words in $\Gamma$.
Furthermore, $\Delta$ is conjugate to the subgroup 
\begin{equation*}
    \Delta' = \left\langle 
    \begin{pmatrix}
    -1 & -4 \\ 0 & -1
    \end{pmatrix},
    \begin{pmatrix}
    5 & 3 \\ -2 & -1
    \end{pmatrix},
    \begin{pmatrix}
    3 & 11 \\ -2 & -7
    \end{pmatrix}
    \right\rangle
\end{equation*}
of $\PSL_2(\bbZ)$ via 
\begin{equation*}
    \gamma'
    = \begin{pmatrix}
    (z^2-3z+2)^{-1/2} & 0 \\ 0 & (z^2-3z+2)^{1/2}
    \end{pmatrix}
\end{equation*}
It follows that $\Delta$ stabilizes some hyperplane $H$ in $\bbH^3$. Since $\bbH^2/\Delta'$ is a finite area twice-punctured torus, the stabilizer of $H$ in $\Gamma$ must act with finite coarea on $H$. Therefore, $M$ contains a totally geodesic surface $S$. 

We now show that this surface is the twice-punctured torus. The trace field of $\Gamma$ has odd degree over $\bbQ$, so $\Gamma$ does not contain any element with purely imaginary trace. Thus $H$ covers an orientable totally geodesic surface in $M$ (see Remark \ref{rem:nonori}). Since the Euler characteristic of the twice-punctured torus is $-2$, the hyperplane $H$ must cover either a once- or a twice-punctured torus in $M$. Consider the element 
\begin{equation}
    d=\begin{pmatrix}
    -1 & 0 \\ 4(z^2-z-1) & -1
    \end{pmatrix} = a^{-1}cb^{-1}c^{-1}b \in \Delta
\end{equation}
Since $wdw^{-1} = x^{-2}\ell$ and $a = x^2\ell$, we have that $a$ and $d$ are not conjugate in $\rho(\Gamma)$. This implies that $a$ and $d$ are not conjugate in $\Delta$. Therefore, $S$ has at least two cusps. It follows that the surface $S$ has to be the twice-punctured torus itself.

Since $a = x^2 \ell$ and $d = w^{-1} x^{-2} \ell w$, the boundary slopes of the two cusps of $S$ are $\pm 2$.
\end{proof}

\begin{remark}
\label{rem:MatricesOfTGS_7_4}
Since $\Delta$ is a conjugate of $\Delta'$ by $\gamma'$, elements in $\Delta$ have the form
\begin{equation*}
    \begin{pmatrix}
    \alpha & \beta(z-1)(z-2) \\ \delta(z^2 - z -1) & \eta  
    \end{pmatrix}
\end{equation*}
where $\alpha,\beta,\delta,$ and $\eta \in \bbZ$.
\end{remark}

\subsection{Boundary slope restrictions}

Since the trace field of $\Gamma$ contains no proper subfield other than $\bbQ$ and $\Gamma$ has integral traces, we apply the trace condition to obtain restrictions for the boundary slopes of totally geodesic surfaces in $M$. 

We first make a few preliminary observations and set some notation. Keeping the notation in Proposition \ref{prop:7_4_TGS_TwicePuncturedTorus}, we denote by $H$ the vertical hyperplane in $\bbH^3$ stabilized by $\Delta$ and by $S$ the totally geodesic twice-punctured torus in $M$ covered by $H$. Since $x^2\ell \in \Delta$, the boundary slope at infinity of $H$ is 2. 
The action of $\Delta$ on $\partial_\infty H$ has two orbits of cusp points. The boundary slopes at $0$ and $\infty$ of $H$ are $-2$ and $2$, respectively, so the two orbits of cusp points are the orbits of $0$ and $\infty$ under $\Delta$. The group $w\Delta w^{-1}$ stabilizes $w(H)$. Since $x^{-2}\ell = wdw^{-1} \in w\Delta w^{-1}$, the image $w(H)$ is a vertical hyperplane with boundary slope $-2$ at infinity.

\begin{lem}
\label{lem:7-4_bdry}
The complete set of boundary slopes for a cusped totally geodesic surface in $M$ is $\{\pm 2\}$.
\end{lem}

\begin{proof}
Let $\Sigma$ be a totally geodesic surface admitting a non-zero boundary slope $p/q$. There is a vertical lift $\widetilde{\Sigma}$ of $\Sigma$ to $\bbH^3$ that contains $\infty$ as a cusp point with boundary slope $p/q$. Suppose that $p/q \neq 2$. This lifts intersects $H$ along a vertical cusp-to-cusp geodesic $(\theta,\infty)$.  Since $\theta$ is a cusp point of $H$, there exists $\gamma \in \Gamma$ such that $\gamma(\infty) = \theta$. In particular, we either have $\theta$ is in the orbit of $0$ or $\infty$ for the action of $\Delta$ on $H$. We consider two cases.

\begin{description}
\item[Case 1] Suppose that $\theta \in \Delta \cdot \{\infty\}$. We can choose $\gamma \in \Delta$ and may assume that
\begin{equation*}
    \gamma = \begin{pmatrix}
    \alpha & \beta (z-1)(z-2) \\
    \delta (z^2-z-1) & \eta
    \end{pmatrix}
\end{equation*}
where $\alpha,\beta,\delta$ and $\eta$ are integers (see Remark \ref{rem:MatricesOfTGS_7_4}). Since $\theta \neq \infty$, we may assume that $\delta \neq 0$. Since $\theta$ is a cusp point of $\widetilde{\Sigma}$, the elements $x^p\ell^q,\gamma x^m\ell^n\gamma^{-1}$ are contained in $\Stab_\Gamma(\widetilde{\Sigma})$. Applying the trace condition, we must have
\begin{align*}
    \tr(x^p\ell^q \gamma x^m \ell^n \gamma^{-1}) &\in \bbZ
\end{align*}
Writing the above expression as an element of $\bbZ[z]$, we see that the trace condition holds if and only if the coefficients of $z$ and $z^2$ are zero; that is, the trace condition is equivalent to
\begin{equation*}
(7mp-6np-6mq-4nq)\delta^2 = 0 = (-3mp-2np-2mq+20nq)\delta^2
\end{equation*}
Since $\delta\neq0$, the only solutions to the above system of equations are $m=n=0$ or $p=q=0$ or $m-2n=p-2q=0$. Since $(m,n) \neq (0,0) \neq (p,q)$, we must have $m/n =p/q = 2$. This contradicts the assumption that $p/q \neq 2$. 

\item[Case 2] Suppose that $\theta \in \Delta \cdot \{0\}$. We can choose $\gamma = \widetilde{\gamma} w^{-1}$ for some $\widetilde{\gamma} \in \Delta $ such that $\widetilde{\gamma}(0) =\theta$. Therefore,
\begin{equation*}
    \gamma = \begin{pmatrix}
    * & * \\
    \delta(z^2-2z) & * \\ 
    \end{pmatrix}
\end{equation*}
Applying the trace condition, we  have
\begin{equation*}
    \delta^2(mp+6np+6mq+4nq)=-4\delta^2(np+mq) = 0
\end{equation*}
Since $\delta \neq 0$ and $(m,n) \neq (0,0) \neq (p,q)$, we must have $m/n =-p/q = 2$ and $m/n = -p/q=-2$. Since $p/q \neq 2$, we must have $m/n =-p/q = 2$. 
\end{description}

Now we suppose that the boundary slope at infinity of $\widetilde{\Sigma}$ is 2. Recall that $w(H)$ is a vertical hyperplane with boundary slope $-2$ at infinity. By Lemma \ref{lem:IntersectionsOfTGS}, the two hyperplanes $\widetilde{\Sigma}$ and $w(H)$ intersect along a cusp-to-cusp geodesic $(\theta,\infty)$. As before, we have two cases.
\begin{description}
\item[Case 1] Suppose that $\theta \in w \Delta \cdot\{\infty\}$. Then there exists $\widetilde{\gamma} \in \Delta$ such that $\theta = w \widetilde{\gamma} (\infty)$. Putting $\gamma =w \widetilde{\gamma} $, we have
\begin{equation*}
    \gamma = \begin{pmatrix}
    * & * \\ \delta(z^2-2z) & * 
    \end{pmatrix}.
\end{equation*}
Since $\gamma(\infty) = \theta$, we have $\gamma x^m \ell^n \gamma^{-1} \in \Stab_{\Gamma}(w(H))$ for some $(m,n) \neq (0,0)$. Applying the trace condition, we have
\begin{equation*}
    tr(x^{2}\ell \gamma x^m \ell^n \gamma^{-1}) \in \bbZ
\end{equation*}
This is equivalent to
\begin{equation*}
    \delta^2 (m+2n) = 0
\end{equation*}
Since $\delta \neq 0$ because $\theta \neq \infty$, we conclude $m/n = -2$. 

\item[Case 2] Suppose that $\theta \in w \Delta \cdot\{0\}$. Then there exists $\widetilde{\gamma} \in \Delta$ such that $\theta = w \widetilde{\gamma}(0)$. Putting $\gamma = w\widetilde{\gamma} w^{-1}$, we have $\gamma(\infty) = \theta$ and furthermore
\begin{equation*}
    \gamma = \begin{pmatrix}
    * & * \\ \delta(z^2-2z-1) & * 
    \end{pmatrix}
\end{equation*}
Applying the trace condition, we have 
\begin{equation*}
    \delta^2 n = \delta^2(m+18n) = 0
\end{equation*}
Similar to before, $\delta \neq 0$ because $\theta \neq \infty$, so we conclude $m=n=0$, which contradicts the assumption that $(m,n)\neq (0,0)$. 
\end{description}
In all cases, we must have both boundary slopes $\pm 2$. Therefore, the complete set of boundary slopes of totally geodesic surfaces in $M$ is $\{\pm 2\}$. 
\end{proof}

\subsection{Uniqueness of the totally geodesic surface}

The method for proving uniqueness is to show that there is no vertical hyperplane of boundary slope $2$ between $H$ and $x(H)$ that is a lift of a totally geodesic surface. In particular, we will argue that any vertical hyperplane of boundary slope $2$ between $H$ and $x(H)$ does not intersect hemispherical lifts of the twice-punctured torus $S$ along closed nor cusp-to-cusp geodesics, which contradicts Lemma \ref{lem:IntersectionsOfTGS}.

\begin{figure}[H]
    \centering
    \includegraphics[scale=0.7,angle=90]{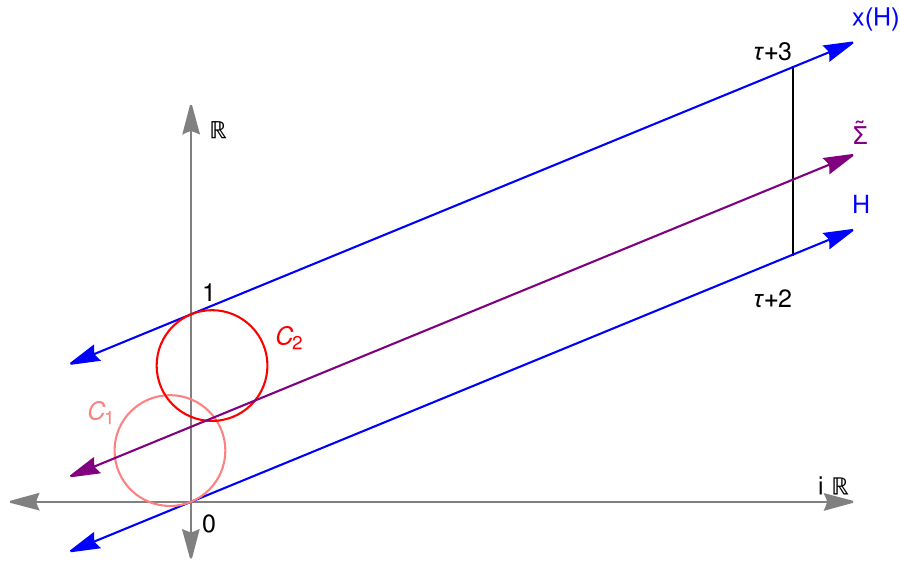}
    \caption{Possible lift of totally geodesic surface with boundary slope $2$ at $\infty$}
    \label{fig:7-4_LiftsOfT}
\end{figure}

\begin{proof}[Proof of Theorem \ref{t:7-4!surf}]
Let $\Sigma$ be a totally geodesic surface in $M$ distinct from the twice-punctured torus $S$.
Let $\widetilde{\Sigma}$ be a vertical hyperplane lift with boundary slope $2$ at $\infty$; such a lift must exist because of Lemma \ref{lem:7-4_bdry}.  Without loss of generality, we may assume that $\partial_{\infty} \widetilde{\Sigma} \cap \partial_{\infty} \bbH_{\bbR}^2 \in (0,1)$ by translation by $x$.

We now describe some hemispherical lifts of $S$ to $\bbH^3$. A visual reference is given in Figure \ref{fig:7-4_LiftsOfT}. Consider two lifts of $S$ defined as $C_1 = y(H)$ and $C_2 = xy^{-1}(H)$. Since $y$ is a parabolic element fixing $0 \in \partial_\infty H$, the boundaries $\partial_\infty C_1$ and $\partial_\infty y^{-1}(H)$ are circles tangent to $\partial_\infty H$ at $0$. Applying the isometry $x$, we see that $\partial_\infty C_2$ is a circle tangent to $x(H)$ at $1$. Finally, we show that $C_1 \cap C_2$ contains two points. Since
\begin{equation*}
    y^{-1}xy^{-1}\left(\frac{\tau +2}{4}\right) = \infty,
    \hspace{4mm}
    y^{-1}xy^{-1}(\infty) = -\frac{\tau+2}{4}, 
    \hspace{4mm}
    y^{-1}xy^{-1}(0) = -\frac{1}{8}\tau-\frac{3}{4},
\end{equation*}
$y^{-1}xy^{-1}(H)$ is a vertical hyperplane with boundary slope $-2$ at $\infty$. Therefore, $y^{-1}xy^{-1}(H)$ intersects $H$ at two points. Applying the isometry $y$, we see that $C_1$ and $C_2$ must also intersect at two points.  

Since $C_1$ and $C_2$ have a nonempty intersection, the vertical hyperplane $\widetilde{\Sigma}$ is not tangent to both $C_1$ and $C_2$ and hence must intersect either $C_1$ or $C_2$ along a geodesic. Let $\tilde\theta$ be a geodesic in the intersections $\widetilde{\Sigma}\cap C_1$ and $\widetilde{\Sigma}\cap C_2$. By Lemma \ref{lem:IntersectionsOfTGS}, $\tilde\theta$ is a lift of either a closed geodesic or a cusp-to-cusp geodesic.  Since the stabilizer of $H$ is conjugate to $\Delta'$ by
\begin{equation*}
    \gamma'
    = \begin{pmatrix}
    (z^2-3z+2)^{-1/2} & 0 \\ 0 & (z^2-3z+2)^{1/2}
    \end{pmatrix}
\end{equation*}
the endpoints $\theta_i$ of $\tilde\theta$ are the image of quadratic irrationals (resp.~rationals) $\sigma_i$ under
\begin{equation*}
    \gamma
    \begin{pmatrix}
    (z^2-3z+2)^{1/2} & 0 \\ 0 & (z^2-3z+2)^{-1/2}
    \end{pmatrix}
\end{equation*}
when $\theta$ is the lift of a closed (resp.~cusp-to-cusp) geodesic for $\gamma \in \{y,xy^{-1}\}$.
Since $\theta_i \in \partial_{\infty} \widetilde{\Sigma}$ and $\widetilde{\Sigma}$ is a vertical hyperplane with boundary slope $2$ at $\infty$, we must have
\begin{equation*}
    \theta_1 - \theta_2 \in \bbR(\tau+2).
\end{equation*}
We recall that $\tau+2 = -4 (z-1)(z-2)$.  Moreover, we note that, for both the cases of the closed and the cusp-to-cusp geodesics, we have $\sigma_1 \sigma_2, \sigma_1 + \sigma_2 \in \bbQ$.
Now we consider the cases of $\gamma = y$ and $\gamma=xy^{-1}$.

\begin{description}
\item[Case 1] Suppose that $\gamma = y$.  Then
\begin{align*}
    \dfrac{\theta_1 - \theta_2}{\tau+2}
    & = \dfrac{1}{-4 (z-1)(z-2)} 
        \cdot \left( y \, \sigma_1 (z-1)(z-2) - y( \, \sigma_2 (z-1)(z-2) \right)
    \\
    & = \dfrac{\sigma_1-\sigma_2}{
        -4 \left( 1 + \sigma_1 z (z-1)(z-2) \right) 
        \left( 1 + \sigma_2 z (z-1)(z-2) \right)
        }
    \in \bbR
\end{align*}
Because $\sigma_1,\sigma_2 \in \bbR$, it is sufficient to verify that the denominator is real; that is, we must have
    \begin{align*}
    \hspace{10mm}
    ( 1 & + \sigma_1  z (z-1)(z-2) ) 
        ( 1 + \sigma_2 z (z-1)(z-2))\\
    & = (\sigma_1 + \sigma_2 - 2 \sigma_1 \sigma_2) z^2
    + (3 \sigma_1 \sigma_2 - 2 (\sigma_1 + \sigma_2)) z
    \\
    & \qquad
    + (1 - \sigma_1 - \sigma_2 + \sigma_1 \sigma_2)
      \in \bbR \cap \bbQ(z) = \bbQ
    \end{align*}
Because $\bbQ(z)$ is cubic, $\sigma_1$ and $\sigma_2$ satisfy the system of the following equations
\begin{equation*}
    \begin{cases}
    \sigma_1 + \sigma_2 - 2 \sigma_1 \sigma_2
    & = 0
    \\
    - 2 (\sigma_1 + \sigma_2) + 3 \sigma_1 \sigma_2
    & = 0
    \end{cases}
\end{equation*}
and the only solution (real or complex) to such a system of equations is $\sigma_1 = \sigma_2 = 0$, which contradicts that $\theta_1$ and $\theta_2$ are distinct endpoints.

\item[Case 2] Suppose that $\gamma = x y^{-1}$.  Then
\begin{align*}
    \dfrac{\theta_1 - \theta_2}{\tau+2}
    & = \dfrac{
        \left( xy^{-1}(\sigma_1 (z-1)(z-2) ) - xy^{-1}(\sigma_2 \cdot (z-1)(z-2)) \right)
    }{-4 (z-1)(z-2)} 
    \\
    & = \dfrac{\sigma_1-\sigma_2}{
        -4 \left( - 1 + \sigma_1 z (z-1)(z-2) \right) 
        \left( - 1 + \sigma_2 z (z-1)(z-2) \right)
        }
    \in \bbR
\end{align*}
Because $\sigma_1,\sigma_2 \in \bbR$, it is sufficient to verify that the denominator is real.  Recalling that $\sigma_1 \sigma_2, \sigma_1+\sigma_2 \in \bbQ$, we must have
    \begin{align*}
    \hspace{8mm}
    ( -1 & + \sigma_1
    z (z-1)(z-2) ) \left( -1 + \sigma_2 z (z-1)(z-2) \right)
    \\
     & = (- \sigma_1 - \sigma_2 - 2 \sigma_1 \sigma_2) z^2
    + (3 \sigma_1 \sigma_2 + 2 (\sigma_1 + \sigma_2)) z
    \\
    & \qquad
    + (1 + \sigma_1 + \sigma_2 + \sigma_1 \sigma_2)
      \in \bbR \cap \bbQ(z) = \bbQ
    \end{align*}
Because $\bbQ(z)$ is cubic, $\sigma_1$ and $\sigma_2$ satisfy the system of the following equations
\begin{equation*}
    \begin{cases}
    - (\sigma_1 + \sigma_2) - 2 \sigma_1 \sigma_2
    & = 0
    \\
    2 (\sigma_1 + \sigma_2) + 3 \sigma_1 \sigma_2
    & = 0
    \end{cases}
\end{equation*}
and, as in Case 1, the only solution (real or complex) to such a system of equations is $\sigma_1 = \sigma_2 = 0$, which again contradicts that $\theta_1$ and $\theta_2$ are distinct endpoints.
\end{description}
Since no such pair of distinct $\theta_i$ exists, $\widetilde{\Sigma}$ cannot be the lift of a totally geodesic surface or else Lemma \ref{lem:IntersectionsOfTGS} would be contradicted.  Therefore, $S$ is the unique totally geodesic surface in $M$.
\end{proof}

\begin{remark}
Note that the degenerate solution $\sigma_1 = \sigma_2$ corresponds to either the tangency between $\partial_{\infty} C_1$ and $\partial_{\infty} H$ or the tangency between $\partial_{\infty} C_2$ and $\partial_{\infty} x(H)$. In other words, the degenerate solution occurs precisely when $\Sigma = S$.
\end{remark}


\section{Balanced pretzel knots}
\label{sec:BalancedPretzelKnots}
In this section, we will give the proof of Theorem \ref{t:pret!surf}. Throughout this section, we let $J = P(2k+1,2k+1,2k+1)$ be the 3-tangle balanced pretzel knot with $2k+1$ half twists in each tangle, $M$ the complement of $J$ in $S^3$, and $\Gamma$ the fundamental group of $M$. 

\subsection{Discrete faithful representation and its trace field}

Following \cite[Proposition 2.1]{Zent}, the knot group $
\Gamma$ has the presentation
\begin{equation}
    \label{eq:BalancedPretzelKnotGroup}
    \Gamma = \langle s_1, s_2, s_3 \mid vs_1 = s_2 v, ws_2 = s_3 w\rangle
\end{equation}
where 
\begin{equation}
\label{eq:vwPretzel}
v = (s_{3}^{-1}s_{2})^k s_{3}^{-1} (s_1 s_{3}^{-1})^k\quad \text{and} 
\quad w = (s_{1}^{-1}s_{3})^k s_{1}^{-1} (s_2 s_{1}^{-1})^k. 
\end{equation}

For $k\geq 1$, $M$ admits a complete hyperbolic structure of finite volume. The discrete faithful representation $\rho:\Gamma \to \SL_2(\bbC)$ sends all conjugates of the meridians of the knot to parabolic isometries.

\begin{thm} 
\label{thm:PretzelTraceField}
The discrete faithful representation $\rho:\Gamma \to \SL_2(\bbC)$ can be conjugated to be of the form
\begin{equation}
\label{eq:PretzelHolonomy}
\rho(s_1) = \begin{pmatrix}
1 & 1 \\ 0 & 1 
\end{pmatrix},\quad 
\rho(s_2) = \begin{pmatrix}
1 & 0 \\ -z_k^2 & 1 
\end{pmatrix}
\quad \text{and}\quad
\rho(s_3) = \begin{pmatrix}
1+z_k & 1 \\ -z_k^2 & 1 -z_k
\end{pmatrix}
\end{equation}
where $z_k$ satisfies a polynomial $\Lambda_k(z) \in \bbZ[z]$. The polynomial $\Lambda_k(z)$ is irreducible and defined recursively by 
\begin{equation}
    \label{eq:LambaRecursion}
    \Lambda_k(z) = (z^2+2) \Lambda_{k-1}(z) - \Lambda_{k-2}(z)
\end{equation}
with the initial conditions $\Lambda_0(z)=z-1$ and $\Lambda_1(z) = z^3 - z^2 + 3z -1$.  
\end{thm}

We show that Corollary \ref{cor:PretzelTraceField} is a consequence of Theorem \ref{thm:PretzelTraceField}.

\begin{proof}[Proof of Corollary \ref{cor:PretzelTraceField}]
By \cite[Lemma 3.5.3]{MR03}, the trace field of $\Gamma$ is generated over $\bbQ$ by
\[
\tr(\rho(s_i)) = 2, 
\hspace{2mm}
\tr(\rho(s_1s_2)) =\tr(\rho(s_2s_3)) = 2-z_k^2, 
\hspace{2mm}
\tr(\rho(s_1s_2s_3)) = 2-3z_k^2 -z_k^3.
\]
Therefore, $\bbQ(\tr \Gamma) = \bbQ(z^2_k,z^3_k) = \bbQ(z_k)$. By Theorem \ref{thm:PretzelTraceField}, $z_k$ satisfies the irreducible polynomial $\Lambda_k$ which implies that the degree of $\bbQ(\tr \Gamma)$ is the degree of $\Lambda_k$. By an inductive argument using the recursive relation in \eqref{eq:LambaRecursion} along with its initial conditions, we see that the degree of $\Lambda_k$ is $2k+1$ which completes the proof of this corollary.
\end{proof}

Before proving the theorem, we make some preliminary observations. Let $\rho:\Gamma \to \PSL_2(\bbC)$ be the discrete and faithful representation of $\Gamma$ coming from the hyperbolic structure on $M$. Since $s_i$ is a meridian generator of $\Gamma$, we can conjugate $\rho(s_1)$ and $\rho(s_2)$ to be the upper and lower triangular matrices stated in the theorem for some $z_k \in \bbC$. The image of $s_3$ under $\rho$ takes the form of a generic conjugate of a parabolic isometry which is given by
\begin{equation*}
    \rho(s_3) = \begin{pmatrix}
    a & b \\ c & d 
    \end{pmatrix}
    \begin{pmatrix}
    1 & 1 \\ 0 & 1
    \end{pmatrix}
    \begin{pmatrix}
    d & -b \\ -c & a
    \end{pmatrix} = 
    \begin{pmatrix}
    1 -ac & a^2 \\ -c^2 & 1 + ac
    \end{pmatrix}
\end{equation*}
where $a,b,c,d \in \bbC$ such that $ad-bc = 1$. 

Note that $M$ admits an order three rotational symmetry that cyclically permutes the twist regions. By Mostow rigidity, this rotational symmetry is homotopic to an order three isometry $r:M \to M$. Observe that $r$ permutes the homotopy class of loops $s_1s_2^{-1}$, $s_2s_3^{-1}$ and $s_3s_1^{-1}$ since they are loops surrounding the twist regions (see Figure \ref{fig:9-35_seifert}). It follows that the elements $s_1s_2^{-1}$, $s_2s_3^{-1}$ and $s_3s_1^{-1}$ are conjugate in the orbifold fundamental group of $M/\langle r \rangle$. Thus, we have $\tr(s_1s_2^{-1}) = \tr(s_2s_3^{-1}) =
    \tr(s_3s_1^{-1}) $ or equivalently
\begin{equation*}
    2+z^2 = 2 +a^2z^2 = 2+c^2.
\end{equation*}
Since $z\neq 0$, these equations imply that $a^2= 1$ and $c^2 = z^2$. Without loss of generality, we can choose $c = z$ and $a=-1$. This shows that $\rho(\Gamma)$ can be conjugated to be of the form as stated in \eqref{eq:PretzelHolonomy}.

Let $F_3$ be the free group on three generators $S_1$, $S_2$, and $S_3$. We consider the surjective homomorphism $\pi:F_3 \to \Gamma$ sending $S_i$ to $s_i$ and the homomorphism $P:F_3 \to \SL_2(\bbZ[z])$ defined by
\[
P(S_1) = \begin{pmatrix}
1 & 1 \\ 0 & 1 
\end{pmatrix},\quad 
P(S_2) = \begin{pmatrix}
1 & 0 \\ -z^2 & 1 
\end{pmatrix},\quad
\text{and}\quad
P(S_3) = \begin{pmatrix}
1+z & 1 \\ -z^2 & 1 - z
\end{pmatrix}.
\]
 
\begin{lem}
\label{lem:PrefixSurfixOfVW}
We have the following identities:
\begin{align*}
    P(S_3^{-1}S_2)^k &=  \begin{pmatrix}
    \alpha_k & \beta_k \\ z^3 \beta_k & \delta_k
    \end{pmatrix},
    \hspace{2mm}
    P(S_1S_3^{-1})^k =  \begin{pmatrix}
    \alpha_k & -z\beta_k \\ -z^2 \beta_k & \delta_k 
    \end{pmatrix},\\
    P(S_1^{-1}S_3)^k &=  \begin{pmatrix}
    \alpha_k-2z\beta_k  & -z\beta_k \\ z^2 \beta_k & \delta_k+2z\beta_k 
    \end{pmatrix},
    \hspace{2mm}
    P(S_2S_1^{-1})^k =  \begin{pmatrix}
    \delta_k + z\beta_k & \beta_k \\ z^2 \beta_k & \alpha_k -z\beta_k 
    \end{pmatrix}.
\end{align*}
The polynomials $\alpha_k,\beta_k$ and $\delta_k$ are defined recursively by:
\begin{equation}
\label{eq:PrefixSurfixOfV}
P(S_3^{-1}S_2)^k = (z^2 + 2)P(S_3^{-1}S_2)^{k-1} - P(S_3^{-1}S_2)^{k-2}
\end{equation}
where $\alpha_0 = \delta_0 = 1$, $\beta_0= 0$, $\alpha_1= z^2-z+1$, $\beta_1 = -1$, $\delta_1 = z+1$. Furthermore, we have
\[
\delta_k = \alpha_k-2\beta_kz +\beta_kz^2.
\]
\end{lem}

\begin{proof}
Applying Cayley--Hamilton, we obtain \eqref{eq:PrefixSurfixOfV}. The initial conditions are obtained by directly compute $P(S_3^{-1}S_2)^k$ when $k=0$ and $k=1$. Similarly, we see that $P(S_1S_3^{-1})^k$, $P(S_1^{-1}S_3)^k$, and $P(S_2S_1^{-1})^k$ all satisfy the same recurrence as in \eqref{eq:PrefixSurfixOfV}. The formulas for $P(S_1S_3^{-1})^k$, $P(S_1^{-1}S_3)^k$, and $P(S_2S_1^{-1})^k$ can be verified by observing that they hold for $k=0$ and $k=1$ and are preserved by the recurrence.  

It remains to check the final identity relating $\alpha_k$, $\beta_k$, and $\delta_k$. Since the identity relating $\alpha_k$, $\beta_k$, and $\delta_k$ is $\bbZ[z]$-linear in $\alpha_k$, $\beta_k$, and $\delta_k$ and $\alpha_k$, $\beta_k$, and $\delta_k$ satisfy the same recursion, it suffices to check that the identity holds for $k=0$ and $k=1$. 
\end{proof}

Let $V$ and $W$ be the lift of $v$ and $w$ using the respective word in \eqref{eq:vwPretzel}. Let us write 
\[
P(V) = \begin{pmatrix}
v_k^{11} & v_k^{12} \\ v_k^{21} & v_k^{22}
\end{pmatrix} 
\quad \text{and} \quad
P(W) = \begin{pmatrix}
w_k^{11} & w_k^{12} \\ w_k^{21} & w_k^{22}
\end{pmatrix} 
\]
where $v^{ij}_k, w_k^{ij} \in \bbZ[z]$. A direct calculation using the identities in Lemma \ref{lem:PrefixSurfixOfVW} gives us
\begin{align}
\label{eq:Ventries}
    &v^{11}_k = (-\beta_kz +\alpha_k)(\beta_kz^2+(\beta_k-\alpha_k)z+\alpha_k),
    \hspace{2mm}
    v^{21}_k = -z^2 v^{12}_k,\\
    \nonumber
    &(w_k^{11}-w^{22}_k)z + w^{21}_k = 0,
    \hspace{2mm}
    w_k^{12} z + w_k^{22} = (-\beta_kz +\alpha_k)(\beta_kz^2+(\beta_k-\alpha_k)z+\alpha_k)
\end{align}
We have the following lemma.

\begin{lem}
\label{lem:RuleOutOneFactor}
Let $\rho$ be the discrete and faithful representation of $\Gamma$ coming from the hyperbolic structure given by \eqref{eq:PretzelHolonomy}. Let $z_k \in \bbC$ and $\eval:\SL_2(\bbZ[z]) \to \SL_2(\bbC)$ be the evaluation map at $z= z_k$. Then $\eval\circ P=\rho \circ \pi$ if and only if $z_k$ is a root of
\[\Lambda_k(z) := \beta_kz^2+(\beta_k-\alpha_k)z+\alpha_k.\]
Furthermore, $\Lambda_k$ can also be defined recursively by 
\[
    \Lambda_k(z) = (z^2+2) \Lambda_{k-1}(z) - \Lambda_{k-2}(z)
\]
with the initial conditions $\Lambda_0(z)=z-1$ and $\Lambda_1(z) = z^3 - z^2 + 3z -1$. 
\end{lem}

\begin{proof}
The map $\eval\circ P$ factors through $\Gamma$ if and only if $z_k$ satisfies 
\[P(VS_1) =P(S_2V), \quad \text{and} \quad P(WS_2) = P(S_3W).\] The first equation is equivalent to $v^{11}_k = 0$ and $v^{21}_k = - z^2v^{12}_k$ while the second equation is equivalent to $(w_k^{11}-w^{22}_k)z + w^{21}_k = 0$ and $w_k^{12} z + w_k^{22}=0$. The calculation prior to the lemma shows that $z_k$ satisfying these equations is equivalent to $z_k$ satisfying
\[
(-\beta_kz +\alpha_k)(\beta_kz^2+(\beta_k-\alpha_k)z+\alpha_k) = (-\beta_kz +\alpha_k)\Lambda_k (z).
\]
To prove this lemma, we must rule out the case that $z_k$ satisfies $-\beta_kz + \alpha_k$. 

For a contradiction, suppose that $z_k$ satisfies $-\beta_kz + \alpha_k$. Note that
\[
    \tr(P((S_1S_3^{-1})^kS_1)) 
= -  \beta_kz^2 + \alpha_k + \delta_k  
= 2(-\beta_k z + \alpha_k ).
\]
If $z_k$ satisfies $-\beta_kz + \alpha_k=0$, then $\rho(\Gamma)$ contains a finite order element. This contradicts the fact that $\rho$ is faithful and $\Gamma$ is torsion-free. As a consequence, $\eval \circ P =\rho \circ \pi$ if and only if $z_k$ satisfies $\Lambda_k(z)=0$. The claim about the recurrence for $\Lambda_k$ follows from the fact that $\alpha_k$ and $\beta_k$ satisfy the same recurrence and that the formula for $\Lambda_k$ in terms of $\alpha_k$ and $\beta_k$ is $\bbZ[z]$-linear. Finally, the initial condition for $\Lambda_k$ is obtained by a direct calculation for $k=0$ and $k=1$. 
\end{proof}

Now we will turn our attention to the irreducibility of $\Lambda_k(z)$. Using the recursion for $\Lambda_k(z)$, we get a closed formula of $\Lambda_k(z)$:
\begin{equation}
    \label{eq:Lambda_z} 
    \Lambda_k(z) = -\sum_{j=0}^k\binom{k+j}{2j}z^{2j} + \sum_{j=0}^{k} \left(\binom{k+j}{2j+1} + \binom{k+j+1}{2j+1}\right) z^{2j+1}
\end{equation}
where any binomial term with larger lower entry evaluates to zero by convention.
The idea for the proof of irreducibility of $\Lambda_k(z)$ is similar to that of irreducibility of the Riley polynomial for twist knots in \cite{HS01}. One explanation for the similarity is that both families of twist knots and balanced pretzel knots are obtained from doing $1/n$ Dehn filling on the Whitehead link and the augmented pretzel link, respectively (see \cite{MMT20}). Both of these links are arithmetic with trace field $\bbQ(i)$ \cite[Theorem 1.2]{MMT20}. Following \cite[Section 3]{HS01}, we consider the substitution $z = x-x^{-1}$. 

Let $\Psi_k(x) = x^{2k+1}\Lambda_k(x-x^{-1})$. Using the recursive formula for $\Lambda_k(z)$, we get
\begin{equation}
    \Psi_k(x) = x^{4k+2} - 1 + \sum_{j=0}^{2k} (-1)^{j+1}x^{2j+1}
\end{equation}

\begin{prop}
\label{prop:RootsOfPsi}
The polynomial $\Psi_k(x)$ has two real roots and $k$ distinct roots in the interior of each quadrant. If $x_0>0$ is the positive real root of $\Psi_k(x)$ and $x_1,\dots,x_k$ are the roots of $\Psi_k(x)$ in the interior of the first quadrant, then $|x_i| > 1$ for all $0 \leq i \leq k$. 
\end{prop}

\begin{proof}
    We consider the auxiliary polynomial 
\begin{equation}
    \Phi_k(x) = (x^2+1) \Psi_k(x) = x^{4k+4}-x^{4k+3} + x^{4k+2} - x^2-x-1.
\end{equation}
We will first study the roots of $\Phi_k$. Note that $x$ satisfies $\Phi_k$ if and only if $x$ satisfies 
\begin{equation*}
    x^{4k+2} - \frac{x^2+x+1}{x^2-x+1} =0.
\end{equation*}
We claim that $\Phi_k(x)$ has exactly one positive real root $x_0 >0$. For convenience, we write
\begin{equation*}
    f(x) = \frac{x^2+x+1}{x^2-x+1} \text{ and } g(x) = x^{4k+2} - f(x).
\end{equation*}
Since $f(x) > 1$ and $x^{4k+2}\leq 1$ for all $0<x<1$, any positive root of $\Phi_k$ must be strictly larger than 1. The derivative of $g(x)$ is 
\begin{equation*}
    g'(x) =(4k+2)x^{4k+1}+ \frac{2(x^2-1)}{(x^2-x+1)^2}.
\end{equation*}
Therefore, $g'(x)$ is positive for all $x \geq 1$ and $g(x)$ is strictly increasing on $[1,\infty)$. Since $g(1) = -2$ and $g(x)$ tends to $+\infty$ as $x$ tends to $\infty$, there exists a unique $x_0 > 1$ such that $g(x_0) = 0$. It follows that $\Phi_k(x)$ has exactly one real positive root $x_0 > 1$. Therefore, $\Phi_k$ and hence $\Psi_k$ has exactly two real roots $\{x_0, -x_0^{-1}\}$.  

We next claim that $\Phi_k$ has exactly $k$ roots in the interior of the second quadrant. We prove this using the argument principle. In particular for any $0\leq n \leq k-1$, let $\gamma_n$ be the sector bounded by
\begin{itemize}
    \item the rays $r_n$ and $r_{n+1}$ where 
    \begin{equation*}
        r_n = \{t e^{i\theta_n} \mid 0\leq t \leq 1\} \text{ and } \theta_n = \frac{\pi}{2} + \frac{2n+1}{4k+4}\pi= \frac{2k + 2n+3}{4k+4} \pi
    \end{equation*}
    \item and the arc
    \begin{equation*}
        c_n = \{e^{it} \mid \theta_n \leq t \leq \theta_{n+1}\}
    \end{equation*}
\end{itemize}
We will show that $\frac{1}{2\pi i } \int_{\gamma_n} d \log(g(x)) = 1$ for all $0\leq n \leq k-1$ where $\log$ has a branch cut along $[0,\infty) \subset \bbR$. To compute this integral, we count the winding number of $g(\gamma_n)$ around $0$.

We first claim that the image of $r_n$ for $0\leq n \leq k$ under $g(x)$ lies in the lower-half plane. The imaginary part of $g(te^{i\theta})$ is 
\begin{equation*}
    \text{Im}(g(te^{i\theta})) = t^{4k+2}\sin((4k+2)\theta) + \frac{2t(t^2-1)\sin\theta}{|t^2e^{i2\theta}-te^{i\theta}+1|^2}. 
\end{equation*}
Since $(2k+2n+1)\pi<(4k+2)\theta_n < (2k+2n+2)\pi$, we have $\sin((4k+2)\theta_n) < 0.$ Furthermore, we have $\pi/2< \theta_n < \pi$ for any $0\leq n \leq k$. It follows that
\[
\text{Im}(g(te^{i\theta_n})) < 0  
\]
for $0< t \leq 1$ and $0 \leq n\leq k$. When $t=0$, $g(0) = - 1$. Therefore the image of $r_n$ for $0\leq n \leq k$ under $g(x)$ starts at $-1$ and remains in the lower-half plane for $0<t<1$. In particular, $g(r_n)$ and $g(r_{n+1})$ do not cross the branch cut $[0,\infty)$ of $\log$. 

Next, we will traverse the circular arc of $\gamma_n$. We have 
\[
\text{Im}(g(e^{it})) = \sin((4k+2)t). 
\]
Observe that the interval $[(4k+2)\theta_n,(4k+2) \theta_{n+1}]$ contains exactly two integer multiple of $\pi$: 
\[
(2k+2n+2)\pi \text{ and } (2k+2n+3)\pi. 
\]
Therefore, $g(c_n)$ intersects the real axis exactly twice. Furthermore, the intersection must be transverse since the derivative of $g(c_n)$ in the $y$-direction is not zero at the intersections. The real part of $g(e^{it})$ is 
\[
\text{Re}(g(e^{it})) = \cos((4k+2)t) - \frac{2\cos t + 1}{2\cos t - 1}
\]
which takes a positive and a negative value when $(4k+2)t $ is equal to 
\[
(2k+2n+2)\pi \text{ and } (2k+2n+3)\pi,
\]
respectively. We have showed that the curve $g(\gamma_n)$ intersects the positive real axis transversely at exactly one point. Therefore, $g(r_n \cup c_n \cup r_{n+1})$ winds about the origin exactly once. This implies that the interior of the second quadrant contains exactly $k$ roots of $g(x)$ which are also exactly $k$ roots of $\Phi_k$ in this quadrant.         

Since $\Phi_k(-x^{-1})=-x^{-4k-4}\Phi_k(x)$, the roots of $\Phi_k$ come in sets of four distinct roots, namely 
\[
\{x,\overline{x},-x^{-1},-\overline{x}^{-1}\},
\] 
except when one of the roots is real or $\pm i$. Since $\Phi_k$ has $k$ distinct roots in the interior of the second quadrant, they account for $4k$ distinct roots of $\Phi_k$. Furthermore, $\Phi_k$ has $\pm i$ as roots and exactly two real roots. All roots of $\Phi_k$ are simple because the degree of $\Phi_k$ is $4k+4$. Since $\Phi_k(x) = (x^2+1)\Psi_k(x)$, the roots of $\Psi_k(x)$ are all roots of $\Phi_k$ except $\pm i$. 

Let $x_1,\dots,x_k$ be $k$ distinct roots in the interior of the first quadrant. To complete the proof, we need to show that $|x_i| > 1$. When $|x| =1$ is a root of $\Phi_k$, $|f(x)| = 1$ implies that $x = \pm i$. Suppose that $|x| <1$ is a root of $\Phi_k$. Without loss of generality, we can assume that $x = re^{i\theta}$ where $ 0 < \theta < \pi/2$. However, we have
\[
|f(x)|^2 = \frac{r^4+r^2+1 +2r^2\cos(2\theta) +\cos \theta(2r+2r^3)}{r^4+r^2+1 +2r^2\cos(2\theta) -\cos \theta(2r+2r^3)} > 1
\]
since $\cos\theta >0$ and $r > 0$. This is the desired contradiction.  
 \end{proof}

\begin{lem}
\label{lem:FactorizationOfPsi}
Either $\Psi(r)$ is irreducible or it factors into exactly two irreducible factors of equal degree; that is, $\Psi_k(r) = p(r)q(r)$ where $q(r) = r^{\deg \Psi_k/2} p(-r^{-1})$. 
\end{lem}

\begin{proof}
Suppose that $\Psi_k(x)$ factors into irreducible factors $
p_1(x)\dots p_s(x)$. We have
\[
\Psi_k(x)= x^{4k+2} \Psi_k(-x^{-1}) =\pm ( x^{d_1}p_1(-x^{-1}))\dots( x^{d_s}p_s(-x^{-1}))
\]
where $d_i$ is the degree of $p_i(x)$. The factoring is unique up to reordering, so for every $1\leq j\leq s$, there exists $1 \leq i \leq s$ such that
\[
p_j(x) = x^{d_i} p_i(-x^{-1})
\]
Observe that if $r_0$ is a root of $\Psi_k(x)$, then so is $-r_0^{-1}$. By Proposition \ref{prop:RootsOfPsi}, $\Psi_k(x)$ has no repeated roots. We either have that $r_0$ and $-r_0^{-1}$ are roots of a unique irreducible factor of $\Psi_k(x)$ or $r_0$ and $-r_0^{-1}$ are roots of distinct irreducible factors that come in pairs $p_j(x) = x^{d_i} p_i(-x^{-1})$. Following \cite{HS01}, suppose that $f(x) = g(x)h(x)$. We call a factor $g(x)$ \emph{complete} if $g(x)= 0$ implies that $g(-x^{-1}) = 0$. Therefore, $\Psi_k(x)$ factors into complete factors or pairs of complete factors. 

We claim that there is no complete factor in any factorization of $\Psi_k(r)$ except for $\Psi_k(r)$. Following Proposition \ref{prop:RootsOfPsi}, we let $x_1,\dots,x_k$ be roots of $\Psi_k(x)$ in the interior of the first quadrant, $x_0>0$ and $-x_0^{-1}$ be real roots of $\Psi_k(x)$. The sum of the roots of $\Psi_k$ is
\[
x_0 -x_0^{-1} + \sum_{j=1}^{k}(x_j + \overline{x}_j -x_j^{-1} -\overline{x}_j^{-1}) =x_0(1-x_0^{-2}) + \sum_{j=1}^{k}2 \rpart(x_i)(1-|x_i|^{-2}) = 1
\]
Since $|x_i| > 1$ and $\rpart(x_i) >0$ for all $0\leq i \leq k$, each summand is a positive number. Therefore, no proper subset of these summands can add to an integer. It follows that no factor of $\Psi_k(x)$ over $\bbZ$ can be complete.

Therefore, $\Psi_k(x)$ is either irreducible or factors into pairs of incomplete factors. If there is more than one pair of incomplete factors, then we may combine one pair of incomplete factors to obtain a complete one. Thus, we can only have one pair of incomplete factors when factoring $\Psi_k(x)$.  
\end{proof}

A consequence of this lemma is that:

\begin{cor}
\label{cor:IrreducibilityLambda}
The polynomial $\Lambda_k(z)$ is irreducible. 
\end{cor}

\begin{proof}
Note that any factoring of $\Lambda_k(z)$ induces a factoring of $\Psi_k(x)$ into complete factors. However, such factoring of $\Psi_k$ is impossible by Lemma \ref{lem:FactorizationOfPsi}. 
\end{proof}

\begin{proof}[Proof of Theorem \ref{thm:PretzelTraceField}]
By Lemma \ref{lem:RuleOutOneFactor}, the assignment in \eqref{eq:PretzelHolonomy} satisfies the group relations and gives a discrete faithful representation if and only if $z_k$ is a root of $\Lambda_k(z)$. Lemma \ref{lem:RuleOutOneFactor} also gives the recursive formula that defines $\Lambda_k(z)$. Corollary \ref{cor:IrreducibilityLambda} gives the irreducibility of $\Lambda_k(z)$ for all $k\geq 1$. 
\end{proof}

\subsection{Boundary slope restriction}
\label{sec:PretzBdry}

We will study totally geodesic surfaces in the complement of balanced pretzel knots following the outline used in Section \ref{sec:74}. Throughout the rest of this section, we will identify $\Gamma$ with its image under the discrete faithful representation given in \eqref{eq:PretzelHolonomy}. We first observe that:

\begin{cor}
\label{cor:PretzelNoClosedTGS}
The complement of $P(2k+1,2k+1,2k+1)$ does not contain any closed totally geodesic surface when $2k+1$ is prime.
\end{cor}

\begin{proof}
Corollary \ref{cor:PretzelTraceField} implies that the degree of the trace field of $\Gamma$ is an odd prime. Therefore, the trace field of $\Gamma$ also contains no proper real subfield besides $\bbQ$. Theorem \ref{thm:PretzelTraceField} gives that $\Gamma$ has integral traces. As a corollary of Proposition \ref{prop:NoClosedTG}, $M$ does not contain any closed totally geodesic surface when $2k+1$ is a prime. 
\end{proof}

We now give a topological and group theoretic description of a totally geodesic Seifert surface in the $M$ and describe the meridian and longitude. The shaded surface $S$ in Figure \ref{fig:pretzel_seifert}
is a Seifert surface of genus 1. It was shown by Adams and Schoenfeld that $S$ is totally geodesic \cite[Example 3.1]{AS05} in $M$. By a direct computation using the knot diagram, $\pi_1(S)$ is conjugate to a subgroup $\Delta \le \Gamma$ generated by 
  \begin{align} \label{eq:SeifertSurfaceGens}
    x & = (s_1 s_2^{-1})^{k + 1} (s_3 s_2^{-1})^{k}
    \\
    y & = (s_2 s_3^{-1})^{k + 1} (s_1 s_3^{-1})^{k}
  \end{align}
It is also convenient to note that 
  \begin{equation*}
    \Delta = \begin{pmatrix}
    z^{-1/2} & 0 \\ 
    0 & z^{1/2}
    \end{pmatrix} \Delta' \begin{pmatrix}
    z^{1/2} & 0 \\ 
    0 &  z^{-1/2}
    \end{pmatrix}
  \end{equation*}
where $\Delta'\le \PSL(2,\bbZ)$ is defined as
  \begin{equation*}
    \Delta' = \left\langle 
    \begin{pmatrix}
    2 & 1 \\ 1 & 1 
    \end{pmatrix},
    \begin{pmatrix}
    0 & -1 \\ 1 & 3
    \end{pmatrix}
    \right\rangle
  \end{equation*}
The longitude of the balanced pretzel knot is the boundary of the Seifert surface $S$ and is given by 
  \begin{equation*}
    \ell = y^{-1}xyx^{-1} = \begin{pmatrix}
    -1 & -\tau \\ 0 & -1 
    \end{pmatrix} 
  \end{equation*}
where $\tau = -6/z$ since $\ell$ is conjugate to
\begin{equation*}
    \begin{pmatrix}
    -1 & 6 \\ 0 & -1
    \end{pmatrix}
\end{equation*}
in $\Delta'$. Let us denote by $H_{\tau}$ the totally geodesic hyperplane in $\bbH^3$ stabilized by $\Delta$. We note that $\ell, xyx^{-1}y^{-1} \in \Delta$ are parabolic isometries fixing $\infty$ and $0$, respectively. Therefore, $H_\tau$ is a vertical hyperplane containing the geodesic $(0,\infty)$. The boundary at infinity of $H_\tau$ is the straight line going through $0$ and $\tau$.  The hyperplane $H_{\tau}$ has boundary slope $0$ because it is stabilized by $\ell$.

\begin{figure}[H]
\centering
\includegraphics[page=1,width=2.3in]{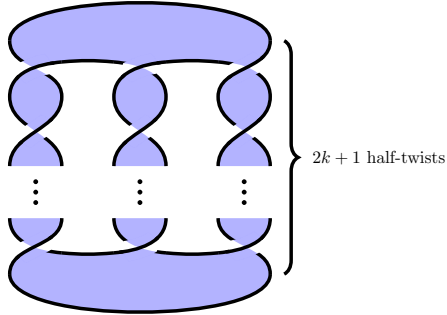}
\caption{The balanced pretzel knot $P(2k+1,2k+1,2k+1)$ with its Seifert surface of genus 1}
\label{fig:pretzel_seifert}
\end{figure}

Similar to the case of $7_4$, we use the trace condition to obtain restrictions on the set of all possible boundary slopes of totally geodesic surfaces in $M$.

\begin{lem} 
\label{lem:pret0}
The complete set of boundary slopes for a cusped totally geodesic surface in the $P(2k+1,2k+1,2k+1)$ balanced pretzel knot complement is $\{0\}$, where $2k+1$ is an odd prime.
\end{lem}

\begin{proof}
The lemma holds for the Seifert surface $S$ since the boundary of $S$ is the homological longitude. By Corollary \ref{cor:PretzelNoClosedTGS}, any totally geodesic surface in $M$ must have at least one cusp. Let $\Sigma$ be a totally geodesic surface admitting a non-zero boundary slope $p/q$. There is a vertical lift $\widetilde{\Sigma}$ of $\Sigma$ to $\bbH^3$ that contains $\infty$ as a cusp point with boundary slope $p/q$. This lift intersects $H_\tau$
along a vertical cusp-to-cusp geodesic $(\theta,\infty)$.  It follows that there are nontrivial elements $s_1^p \ell^q$ and $\gamma s_1^m \ell^n \gamma^{-1}$ in $\Stab_{\Gamma} (\widetilde{\Sigma})$ where $\gamma \in \Gamma$ has the property that $\gamma^{-1}(\theta) = \infty$. Since $S$ has only one cusp and $\theta$ is a cusp point of $H_\tau$, we may choose $\gamma \in \Delta$. Since $\Delta$ is a conjugate of $\Delta'$ by the matrix
\begin{equation*}
    \begin{pmatrix}
    z^{-1/2} & 0 \\ 
    0 & z^{1/2}
    \end{pmatrix}
\end{equation*}
we may assume that $\gamma$ has the form
\begin{equation*}
\gamma = \begin{pmatrix} * & * \\ \delta z & * \end{pmatrix}
\end{equation*}
where $\delta \in \bbZ$. The cusp points $\theta$ and $\infty$ are distinct, so $\gamma$ cannot fix $\infty$.  Consequently, $\delta \neq 0$. 

By Corollary \ref{cor:PretzelTraceField}, the trace field of $\Gamma$ has prime degree $2k+1$ over $\bbQ$ and therefore contains no proper real subfield besides $\bbQ$. Since $\Gamma$ has integral traces, the traces of $\pi_1(\Sigma)$ must be contained in $\bbZ$. Applying the trace condition, we have
  \begin{align*}
  \tr(s_1^p \ell^q \gamma s_1^m \ell^n \gamma^{-1})
  & \in \bbZ
  \\
  \delta^2 (-nq +(np + mq)z -mpz^2)
  & \in \bbZ
  \end{align*}
since $z \tau = -6$. We now compare the coefficients of the powers of $z$, all of which must be equal to $0$ except in the case of the constant term. Since $(p,q) \neq (0,0) \neq (m,n)$, we can manually check that $p=m=0$ is the only solution.
\end{proof}

As a consequence of Lemma \ref{lem:pret0}, Lemma \ref{lem:IntersectionsOfTGS} and Corollary \ref{cor:PretzelNoClosedTGS}, if $2k+1$ is a prime and $\Sigma$ is a totally geodesic surface of $M$ that is not isotopic to $S$, then any vertical lift $\widetilde{\Sigma}$ of $\Sigma$ to $\bbH^3$ must be parallel to vertical lifts of $S$ to $\bbH^3$. Therefore, $\Sigma$ and $S$ can only be disjoint or intersect each other along a union of closed geodesics. We first prove that $\Sigma$, if exists, must indeed intersect $S$ along a union of closed geodesics. To this end, we study the geometric configuration formed by a finite collection of lifts of $S$ to $\bbH^3$.  
We consider the following elements of $\Gamma$ for $0\leq j \leq 2k$:
\[
g_j = \begin{cases}
(s_2s_1^{-1})^{(j-1)/2}s_2 & j \text{ odd}\\
(s_2s_1^{-1})^{j/2} & j \text{ even}
\end{cases}, 
\hspace{4mm}
h_j = \begin{cases}
(s_1s_3^{-1})^{(j+1)/2} & j \text{ odd}\\
(s_1s_3^{-1})^{j/2}s_1 & j \text{ even}
\end{cases}.
\]
Let us denote by 
\[
C_j := g_j(H_\tau), \quad D_j :=  h_j(H_\tau)
\]
the lifts of $S$ to $\bbH^3$. We first observe that the collections $C_j$ and $D_j$ are preserved by an order two symmetry of $M$. 

\begin{lem}
\label{lem:PretzelOrderTwoSymmetry}
Let 
\[
\sigma = \begin{pmatrix}
i & i \, \frac{1-z}{z} \\ 
0 & -i
\end{pmatrix}
\]
Then $\sigma$ is an order two rotational matrix about the geodesic between $\displaystyle\frac{z-1}{2z}$ and $\infty$ which induces an isometry on $M$ such that
\[
\sigma C_j = D_j.
\]
\end{lem}

\begin{proof}
By matrix multiplication, $\sigma^2 = -I_2 \equiv_{\PSL_2(\bbC)} I_2$ which gives us an order $2$ rotation, and its fixed point set is $\{\frac{z-1}{2z},\infty\}$.  To see that $\sigma$ induces a nontrivial isometry of $M$, it is sufficient to show that $\sigma$ normalizes $\Gamma$ and $\sigma \not\in \Gamma$.
We already know that $\sigma \not\in \Gamma$ because $\Gamma$ is torsion-free, so we now verify that $\sigma$ normalizes $\Gamma$ by checking each generator:
\begin{align*}
  \sigma s_1 \sigma^{-1}
  & = \begin{pmatrix} 1 & -1 \\ 0 & 1 \end{pmatrix}
    = s_1 \in \Gamma
  \\
  \sigma s_2 \sigma^{-1}
  & = \begin{pmatrix} z^2-z+1 & -(z-1)^2 \\ z^2 & z^2+z+1 \end{pmatrix}
    = s_1 s_3^{-1} s_1^{-1} \in \Gamma
  \\
  \sigma s_3 \sigma^{-1}
  & = \begin{pmatrix} z^2+1 & -z^2 \\ z^2 & -(z^2-1) \end{pmatrix}
    = s_1 s_2^{-1} s_1^{-1} \in \Gamma
\end{align*}
Using the above calculation, we observe that $\sigma g_j\sigma = h_j s_1^{-1}$. Since $\sigma(H_\tau) = s_1(H_\tau)$, we have
\begin{align*}
    \sigma g_j (H_\tau) = \sigma h_j s_1^{-1} \sigma (H_\tau) = h_j(H_\tau).
\end{align*}
This completes the proof of the lemma.
\end{proof}

We give a description of the configuration of these lifts as follows.
\begin{prop}
\label{prop:PretzelLiftsOfS}
Suppose that $2k+1$ is a prime number. The collection $C_j$ and $D_j$ for $0\leq j \leq 2k$ satisfy the following:
\begin{enumerate}
    \item \label{item:CjDjAreCircles} The collection $\partial_\infty C_j$ and $\partial_\infty D_j$ are all circles for all $0<j\leq 2k$.
    \item \label{item:CjDjTangent} The circles $\partial_\infty  C_j$ and $\partial_\infty  C_{j+1}$ are tangent for all $0 \leq j \leq 2k-1$. Similarly, $\partial_\infty  D_j$ and $\partial_\infty  D_{j+1}$ are tangent for all $0 \leq k \leq 2k-1$. Finally, $\partial_\infty  C_{2k}$ and $\partial_\infty  D_{2k}$ are tangent. 
    \item \label{item:CjIntersections} The tangent point between the circles in $\partial_\infty  C_j$ are
    \[
    \partial_\infty C_j \cap 
    \partial_\infty C_{j+1} = \begin{cases}
    \{g_j(\infty)\} & \text{ if $j$ is odd}\\
    \{g_j(0)\} & \text{ if $j$ is even}
    \end{cases}
    \]
    where $0\leq j \leq 2k$ and $C_{2k+1} := D_{2k}$.
\end{enumerate}
\end{prop}

\begin{proof}
By Lemma \ref{lem:PretzelOrderTwoSymmetry}, $\sigma(C_j) = D_j$. Since $\sigma$ fixes $\infty$, vertical (resp.~hemispherical) hyperplanes remain vertical (resp.~hemispherical) under $\sigma$.  It thus suffices to prove \ref{item:CjDjAreCircles} for $C_j$. For contradiction, suppose that $g_j(H_\tau) = s_1^m(H_\tau)$ for some $m \in \bbZ$; equivalently $s_1^{-m}g_j \in \Delta$ for some $m \in \bbZ$. It is evident from \eqref{eq:SeifertSurfaceGens} that $\Delta \subset [\Gamma,\Gamma]$ since the generators $s_i$ are all conjugate to each other. This implies that $m=0$ when $j$ is even and $m=1$ when $j$ is odd. After possibly conjugating by $s_1$, we have $\tr((s_2s_1^{-1})^r) \in \bbZ$ for $1\leq r \leq k$. By Lemma \ref{lem:PrefixSurfixOfVW},
\[
\tr((s_2s_1^{-1})^r) = \alpha_r + \delta_r = 2\alpha_r -2\beta_r z +\beta_rz^2.
\]
Using the recursion in \eqref{eq:PrefixSurfixOfV} and the initial conditions for $\alpha_r$ and $\beta_r$, we see that the leading terms of $\alpha_r$ and $\beta_r$ are $z^{2r}$ and $-z^{2r-2}$, respectively. The leading term of $\tr((s_2s_1^{-1})^r)$ is $z^{2r}$ for $1  < r \leq k$. Since $z$ has degree $2k+1$ over $\bbQ$, the trace $\tr((s_2s_1^{-1})^r)$ is not an integer, which contradicts the fact that $s_1^{-m}g_j\in \Delta$. Therefore, $\partial_\infty  C_j$ and $\partial_\infty  D_j$ are all circles for $0<j\leq 2k$. 

To prove \ref{item:CjDjTangent}, we first observe that the parabolic isometry $s_2$ fixes $0\in \partial_\infty(H_\tau) = \partial_\infty C_0$. Since $s_2 \not\in\Delta$, it follows that $\partial_\infty C_1 = \partial_\infty s_2(H_\tau)$ is tangent to $\partial_\infty  C_0$ at $0$. For arbitrary $j$, $C_j$ and $C_{j+1}$ are translate of either $C_0$ and $C_1$ when $j$ is even or $C_0$ and $s_1^{-1}(H_\tau)$ when $j$ is odd. Since $\partial_\infty C_0$ and $\partial_\infty s_1^{-1} H_\tau$ are tangent at $\infty$, the circles $\partial_\infty C_j$ and $\partial_\infty C_{j+1}$ are always tangent. Since $\sigma C_j = D_j$ and $\sigma$ preserves tangency of circles in $\partial_\infty \bbH^3$, we see that $\partial_\infty D_j$ and $\partial_\infty D_{j+1}$ are also tangent. For the tangency between $\partial_\infty C_{2k}$ and $\partial_\infty D_{2k}$, it suffices to show that $\partial_\infty C_{2k} \cap \partial_\infty D_{2k} \neq \emptyset.$ The claim then follows because $C_{2k}$ and $D_{2k}$ are both lifts of an embedded surface $S$. From Lemma \ref{lem:RuleOutOneFactor}, we have
\begin{equation*}
\Lambda_k = \beta_k z^2 + (\beta_k -\alpha_k)z+\alpha_k = 0
\end{equation*}
which, combined with Lemma \ref{lem:PrefixSurfixOfVW}, implies that
\[g_{2k}(0)  = \frac{\beta_k}{\alpha_k-z\beta_k} = \frac{z-1}{2z}.\]
In other words, $\partial_\infty(C_{2k})$ contains the fixed point of $\sigma$. Since $\sigma(C_{2k}) = D_{2k}$, the intersection $\partial_\infty(C_{2k}) \cap \partial_\infty(D_{2k})$ contains $\frac{z-1}{2z}$.
\end{proof}

\begin{remark}
\label{rem:PretGeometricConfigOfCjDj}
Proposition \ref{prop:PretzelLiftsOfS} implies that the collection of circles $\partial_\infty C_j $ and $\partial_\infty D_j$ form a chain of sequentially tangent circles lying in the bounded region in $\bbC \subset \partial_\infty \bbH^3$ defined by $\partial_\infty H_\tau$ and $\partial_\infty s_1(H_\tau)$. Experimentally, we observe that these circles must be either tangent or disjoint from each other for all $1 \leq k \leq 10$ (see Figure \ref{fig:9_35_LiftsOfS} for the case $k=1$). Proposition \ref{prop:PretzelLiftsOfS} does not quite give a proof of this picture. However, we do not need this fact anywhere in our results.

\begin{figure}[H]
    \centering
    \includegraphics[scale=0.7,angle=90]{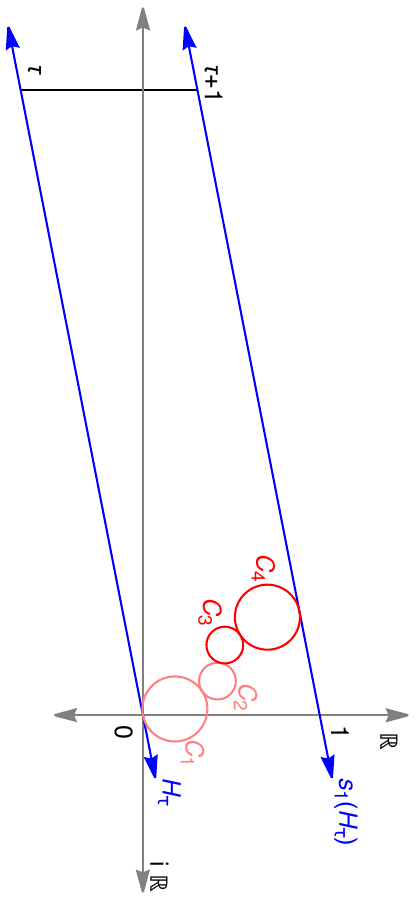}
    \caption{The boundary at infinity of lifts of S to $\bbH^3$ in the pretzel knot $P(3,3,3)$}
    \label{fig:9_35_LiftsOfS}
\end{figure}
\end{remark}

\begin{prop}\label{prop:2pts}
Let $2k + 1$ be a prime number. If $\Sigma$ is a totally geodesic surface in $M$ that is different from $S$, then the intersection $\Sigma \cap S$ is a nonempty union of closed geodesics.
\end{prop}

\begin{proof}
By Corollary \ref{cor:PretzelNoClosedTGS} and Lemma \ref{lem:pret0}, the surface $\Sigma$ has at least one cusp, and the cusps of $\Sigma$ all have boundary slope $0$. In particular, the lifts of $\Sigma$ to $\bbH^3$ at a given cusp point are parallel to those of $S$ at that cusp point. Consequently, $\Sigma \cap S$ cannot contain any cusp-to-cusp geodesic. It remains to show that $\Sigma \cap S$ is not empty. 

Let $\widetilde{\Sigma}$ be a vertical lift of $\Sigma$. Up to the action of $s_1$, we may assume that $\widetilde{\Sigma}$ lies between $H_\tau$ and $s_1(H_\tau)$. In fact, by the previous discussion, $\widetilde{\Sigma}$ is parallel to $H_\tau$ and $s_1(H_\tau)$. By Remark \ref{rem:PretGeometricConfigOfCjDj}, $\partial_\infty(\widetilde{\Sigma}) $ has a nonempty intersection with $\partial_\infty C_j$ or $\partial_\infty D_j$ for some $j$. We want to show that $\partial_\infty(\widetilde{\Sigma})$ intersects one of the circles $\partial_\infty C_j$ and $\partial_\infty D_j$ along two points. If $\partial_\infty \widetilde{\Sigma} \cap \partial_\infty D_j \neq \emptyset$ for some $j$ then $\partial_\infty \sigma(\widetilde{\Sigma}) \cap \partial_\infty C_j\neq \emptyset$ where $\sigma(\widetilde{\Sigma})$ is a lift of a (possibly) different totally geodesic  surface in $M$. Since the number of intersection points in the visual boundary is preserved by $\sigma$, we may assume that $\partial_\infty\widetilde{\Sigma} \cap \partial_\infty C_j \neq \emptyset$ for some $1\leq j \leq 2k$. 

For a contradiction, suppose that $|\partial_\infty \widetilde{\Sigma} \cap\partial_\infty C_j | = 1$ for some $1 \leq j \leq 2k$. By Remark \ref{rem:PretGeometricConfigOfCjDj}, the intersection point $\partial_\infty \widetilde{\Sigma} \cap\partial_\infty C_j$ must be the tangent point between the circles in $\partial_\infty \bbH^3$ described in Proposition \ref{prop:PretzelLiftsOfS}; that is,
\begin{equation*}
\partial_\infty \widetilde{\Sigma} \cap\partial_\infty C_j = \begin{cases}
g_j(\infty) &\text{ if $j$ is odd}\\
g_j(0) &\text{ if $j$ is even}
\end{cases}
\end{equation*}
Let $j$ be odd.  Since $\widetilde{\Sigma}$ has boundary slope $0$ with respect to the coordinates $\{1,\tau\}$, the line between the center of $C_j$ and the tangent point $g_j(\infty)$ must have boundary slope $1/0$; that is, we require
  \begin{equation*}
  \mathrm{cen}(C_j) - g_j(\infty) = it \, \tau
  \end{equation*}
for some $t \in \bbR$.  By direct calculation with $g_j = \begin{pmatrix} a & b \\ c & d \end{pmatrix}$,
  \begin{align*}
  \mathrm{cen}(C_j) - g_j(\infty) 
  & = \dfrac{a \tau \bar{d}-b \bar{c} \bar{\tau}}{c \tau \bar{d}-d \bar{c} \bar{\tau}}
     - \dfrac{a}{c}
  \\
  & = \dfrac{c(a \tau \bar{d}-b \bar{c} \bar{\tau})-a(c \tau \bar{d}-d \bar{c} \bar{\tau})}{c(c \tau \bar{d}-d \bar{c} \bar{\tau})}
  \\
  & = \dfrac{\bar{c} \bar{\tau}}{c}
  \cdot \dfrac{1}{c \tau \bar{d}-d \bar{c} \bar{\tau}}
  \\
  & = \dfrac{1}{c^2 \tau}
  \cdot \dfrac{c \bar{c} \tau \bar{\tau}}{c \tau \bar{d}-d \bar{c} \bar{\tau}}
  \end{align*}
The second fraction is purely imaginary, so it suffices to check when $\frac{1}{c^2 \tau} = t \, \tau$ for some $t \in \bbR$.  Since $\frac{1}{c^2\tau^2} \in \bbQ(z) \setminus \{0\}$, this would imply that $t^2 \in \bbQ(z) \setminus \{0\} $. The number field $\bbQ(z)$ has odd degree over $\bbQ$ and so cannot have an even extension.  Thus $t \in \bbQ(z) \cap \bbR = \bbQ$; equivalently $c \tau \in \bbQ$.  By Lemma \ref{lem:PrefixSurfixOfVW}, 
    \begin{align*}
        c
        & = z^2 \alpha_{(j-1)/2} - z \beta_{(j-1)/2} + z^3 \beta_{(j-1)/2}
        \\
        c \tau
        & = z (z \alpha_{(j-1)/2} - \beta_{(j-1)/2} + z^2 \beta_{(j-1)/2}) \tau
        \\
        & = -6 (z \alpha_{(j-1)/2} - \beta_{(j-1)/2} + z^2 \beta_{(j-1)/2})
    \end{align*}
But $0 < \deg (z \alpha_{(j-1)/2} - \beta_{(j-1)/2} + z^2 \beta_{(j-1)/2}) = j < 2k$, so $\delta \tau$ cannot be real.

An analogous calculation and contradiction can be done for $j$ even.  Therefore for some $j$, the set $\partial_\infty \widetilde{\Sigma} \cap \partial_\infty C_j$ contains two distinct points. In other words, $\Sigma$ and $S$ intersect each other along a nonempty union of closed geodesics. 
\end{proof}

\subsection{Uniqueness of the totally geodesic surface in $P(3,3,3)$}

For this section, let $J$ be the knot $9_{35}$, which is the balanced pretzel knot $P(3,3,3)$.  Let $M = S^3 \setminus J$ and $\Gamma = \pi_1(M)$ with the discrete, faithful representation $\rho:\pi_1(M) \to \PSL_2(\bbC)$ as given in Theorem \ref{thm:PretzelTraceField}.

\begin{figure}[H]
\centering
\includegraphics[width=2in]{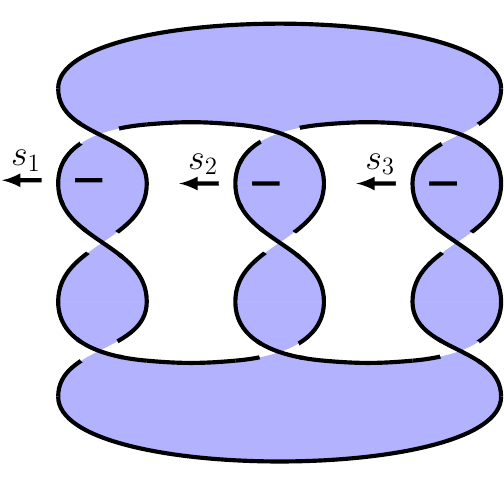}
\caption{Pretzel knot $P(3,3,3)$ with Seifert surface of genus 1}
\label{fig:9-35_seifert} 
\end{figure}

\begin{proof}[Proof of Theorem \ref{t:pret!surf}]
Proposition \ref{prop:2pts} gives the first half of Theorem \ref{t:pret!surf}. We now move onto the proof of uniqueness of the totally geodesic surface $S$ in $M$. This is done by considering possible the intersection between a candidate surface $\Sigma$ with the totally geodesic Seifert surface $S$.

Let $\Sigma \subset M$ be a totally geodesic surface that is different from $S$.  The trace field of $J$ is $\bbQ(z)$ where $z$ has minimal polynomial $z^3 - z^2 + 3z - 1$.  By Proposition \ref{prop:NoClosedTG}, the surface $\Sigma$ cannot be compact, and $\tr(\rho(\pi_1(\Sigma))) \subset \bbZ$ where $\rho$ is the geometric representation of $\Gamma \to \SL_2(\bbC)$.  By Lemma \ref{lem:pret0}, every vertical hyperplane lift of $\Sigma$ must have boundary slope $0$ in coordinates $\{1,\tau\}$.  It is sufficient to show that no such lift $\widetilde{\Sigma}$ exists.

Up to a translation by a power of $s_1$ and a rotation by the element $\sigma$ in Lemma \ref{lem:PretzelOrderTwoSymmetry}, it suffices to show that there is no such lift whose boundary intersects the Euclidean interval $(0,\frac{z-1}{2z})$.  By combining Lemma \ref{lem:IntersectionsOfTGS} and Proposition \ref{prop:2pts}, any such $\widetilde{\Sigma}$ must intersect some $g_j(H_{\tau})$ along the lift of a closed geodesic whose endpoints in $\hat{\bbC} = \partial_{\infty} \bbH^3$ are some $\theta_1$ and $\theta_2$. Recall from the proof of Corollary \ref{cor:PretzelNoClosedTGS} that
\[
\Delta = \begin{pmatrix}
z^{-1/2} & 0 \\ 0 & z^{1/2}
\end{pmatrix}
\Delta' 
\begin{pmatrix}
z^{1/2} & 0 \\ 0 & z^{-1/2}
\end{pmatrix}
\]
where $\Delta'\leq \PSL_2(\bbZ).$ It follows that $\theta_1$ and $\theta_2$ are images of fixed points of hyperbolic elements in $\Delta'$ under 
\[
\gamma' = \begin{pmatrix}
z^{-1/2} & 0 \\ 0 & z^{1/2}
\end{pmatrix}
\]
Since fixed points of hyperbolic elements in $\Delta'$ are quadratic irrationals, there must be some distinct quadratic irrationals $\sigma_1,\sigma_2 \in \bbR$ such that $\theta_1 = g_j \gamma'(\sigma_1)$ and $\theta_2 = g_j \gamma'(\sigma_2)$.
Because $\theta_1,\theta_2 \in \partial_{\infty} \widetilde{\Sigma}$, the $0$ boundary slope requires that $\frac{\theta_2 - \theta_1}{\tau} \in \bbR$; equivalently, $z(\theta_2-\theta_1) \in \bbR$.  These two requirements together means that we must satisfy
  \begin{equation}
  z \left( 
    g_j \big( \frac{\sigma_2}{z} \big) 
  - g_j \big( \frac{\sigma_2}{z} \big)
  \right) \in \bbR
  \end{equation}
for quadratic irrationals $\sigma_1,\sigma_2$.  Recall the minimal polynomial $\Lambda_1(z) = z^3 - z^2 + 3z - 1$ for $z$.

\begin{description}
  \item[Case $j$ odd] If $j$ is odd, then
    \begin{align*}
    z \left( g_j \left(\frac{\sigma_2}{z}\right) 
        - g_j \left( \frac{\sigma_2}{z} \right) \right)
    & = z \left( \dfrac{a_j \sigma_1 + b_j z}{z(a_j + z(b_j-a_j) \sigma_1)} - \dfrac{a_j \sigma_2 + b_j z}{z(a_j + z(b_j-a_j) \sigma_2)} \right)
    \\
    & = \dfrac{\sigma_2-\sigma_1}{(a_j + z(b_j-a_j) \sigma_1)(a_j + z(b_j-a_j) \sigma_2)}.
    \end {align*}
  The numerator is already real, so we need to check whether or not the denominator is real for $j=1$:
    \begin{align*}
    (a_1 + z(b_1-a_1) \sigma_1)(a_1 + z(b_1-a_1) \sigma_2)
    = \sigma_1 \sigma_2 z^2 - (\sigma_1 + \sigma_2) z + 1.
    \end{align*}
  Because $z$ has a cubic minimal polynomial, we must have $\sigma_1 \sigma_2 = \sigma_1 + \sigma_2 = 0$.  There is no real solution for such $\sigma_1$ and $\sigma_2$.  Thus the geodesic is not closed, so the vertical hyperplane $g_1(H_{\tau})$ cannot be the lift of a totally geodesic surface.
  
  \item[Case $j$ even] If $j$ is even, then
    \begin{align*}
    z \left( g_j \left(\frac{\sigma_2}{z}\right) 
        - g_j \left( \frac{\sigma_2}{z} \right) \right)
    & = z \left( \dfrac{a_j \sigma_1 + b_j z}{z(a_j - b_j z^2 + b_j z \sigma_1)} - \dfrac{a_j \sigma_2 + b_j z}{z(a_j - b_j z^2 + b_j z \sigma_2)} \right)
    \\
    & = \dfrac{\sigma_2-\sigma_1}{(a_j - b_j z^2 + b_j z \sigma_1)(a_j - b_j z^2 + b_j z \sigma_2)}.
    \end{align*}
  As in the case above, the numerator is already real, so we need to check whether or not the denominator is real for $j=2$.
    \begin{align*}
    (a_2 & - b_2 z^2 + b_2 z \sigma_1)(a_2 - b_2 z^2 + b_2 z \sigma_2) \\
    & = \, z^4 - (\sigma_1+\sigma_2) + z^2(\sigma_1\sigma_2) + 1
    \\
    & = \, z^2 + 2z(\sigma_1+\sigma_2) + 2-(\sigma_1+\sigma_2) + \big( z+1-(\sigma_1 + \sigma_2) \big) \Lambda_1(z).
    \end{align*}
  This expression is quadratic modulo $\Lambda_1(z)$.  Because $z$ is cubic, this cannot be real.  Thus the geodesic is not closed, so the vertical hyperplane $g_2(H_{\tau})$ cannot be the lift of a totally geodesic surface.
\end{description}
Therefore, there is no totally geodesic surface in the $P(3,3,3)$ knot complement other than the aforementioned Seifert surface.
\end{proof}

The above proof techniques can be applied to $P(2k+1,2k+1,2k+1)$ for higher $2k+1$ odd prime, but the use of $\Lambda_k(z)$ becomes increasingly cumbersome as $k$ and $j$ grow.  We have, however, verified that the theorem and its proof hold for $P(5,5,5)$, $P(7,7,7)$, and $P(11,11,11)$.


\section{Seifert surfaces and the Euler class obstruction}
\label{sec:SeifertSurfacesEulerClassObstructions}

\subsection{Calculating the Euler class}
\label{subsec:CalculatingEulerClass}

Let $M$ be a hyperbolic knot complement. Suppose that the trace field of $M$ admits a real place. Then there exists $\rho:\pi_1(M) \to \PSL_2(\bbR)$ which is a Galois conjugate of the geometric representation of $\pi_1(M)$. Associated to the aforementioned representation, there exists a relative Euler class $e_\rho \in H^2(M,\partial M;\bbZ)$ which defines a norm on $H_2(M,\partial M;\bbZ)$. Let $F$ be a Seifert surface of $M$ and $[F]$ be the homology class of $F$ in $H_2(M,\partial M;\bbZ)$. Since $H_2(M,\partial M;\bbZ)\cong \bbZ$, the norm $e_\rho$ on $H_2(M,\partial M;\bbZ)$ is determined by $e_\rho([F])$. On the one hand, as discussed in Section \ref{subsec:EulerClass}, the representation $\rho$ lifts to a representation $\widetilde{\rho}: \pi_1(M) \to \widetilde{\PSL}_2(\bbR)$. On the other hand, the canonical section of $E_\rho|_{\partial M}$ determines a canonical section $s$ of $\rho|_{\pi_1(\partial M)}$ defined in \eqref{eq:CanSection}. The elements $s(\ell)$ and $\widetilde{\rho}(\ell)$ differ by a central element $c^n$ where $c$ is a generator of the center of $\widetilde{\PSL}_2(\bbR)$. Then $e_\rho([F]) = \pm n$.  

We have a group isomorphism $\psi: \PSL_2(\bbR) \to \PSU(1,1)$ given by
\[
\begin{pmatrix}
a & b \\ c & d\\
\end{pmatrix} \mapsto
\frac{1}{2}\begin{pmatrix} 
a+d+(b-c)i & a-d-(b+c)i \\ a-d-(b-c)i & a+d-(b-c)i
\end{pmatrix}
\]
The group $\PSU(1,1)$ has the following descriptions. As a quotient of a matrix group, $\PSU(1,1)$ is 
\[
\left\{\begin{pmatrix}
\alpha & \beta \\ \overline{\beta} & \overline{\alpha} 
\end{pmatrix} ~\middle|~ |\alpha|^2 - |\beta|^2 = 1\right\} / \{\ \pm I \}.
\]
Setting $\gamma = \overline{\beta}/\alpha$ and $\omega = \arg(\alpha) \mod \pi$, we get a second description of $\PSU(1,1)$ as an open solid torus
\[
  \left\{
  (\gamma,\omega) \in \bbC \times \bbR
  \mid|
  |\gamma|<1, 0\leq \omega \leq \pi \right\}
  /(\gamma,0) \sim (\gamma,\pi)
\]
with the group multiplication
\begin{equation}
\label{eq:GroupLawUniCovPSL2R}
  \begin{array}{r@{~}l}
    (\gamma_1,\omega_1) \cdot (\gamma_2,\omega_2)
  = & \left(
  \dfrac{\gamma_1 + \gamma_2 \exp(-2 i \, \omega_1)}{1 + \gamma_2 \overline{\gamma_1} \exp(-2 i \, \omega_1)}, 
  \,
  \omega_1 + \omega_2 
  \right.
  \\
  & \left. 
  \quad
  +  \dfrac{1}{2i} \log \left( 
   \dfrac{1 + \gamma_2 \overline{\gamma_1} \exp(-2 i \, \omega_1)}{1 + \gamma_1 \overline{\gamma_2} \exp(2 i \, \omega_1)}
   \right)\mod \pi
   \right)
   \end{array}
\end{equation}
where $\overline{\gamma}$ denotes the usual complex conjugate of $\gamma$ and $\log z$ is defined with the branch cut along $(-\infty,0]$.  

From the latter description of $\PSU(1,1)$, we can see that the universal covering group $\widetilde{\PSL}_2(\bbR) = \widetilde{\PSU}(1,1)$ can be described as an infinite open cylinder
\[
  \left\{ (\gamma,\omega) \in \bbC \times \bbR \mid| |\gamma|<1, -\infty < \omega < \infty \right\}
\]
with the group multiplication given by \eqref{eq:GroupLawUniCovPSL2R} where the second coordinate is no longer computed modulo $\pi$. The center of $\widetilde{\PSL}_2(\bbR)$ is the subgroup generate by $c = (0,\pi)$ which is given explicitly by
\[
\{(0,k \pi)\mid k \in \bbZ \}
\]
Up to conjugation, let us assume that $\rho(\pi_1(\partial M))$ is contained in the parabolic subgroup of $\PSL_2(\bbR)$ fixing $\infty \in \partial_\infty\bbH^2$. Let $\ell$ be the longitude of the knot. If the image $\rho(\ell)$ itself is a matrix \[
\begin{pmatrix} -1 & -\tau \\ 0 & -1 \end{pmatrix},
\]
then under the canonical section
\[
s(\ell) = \left(
     \dfrac{i \, \tau}{2+i \, \tau}, \arctan \left( \dfrac{\tau}{2} \right)
  \right).
\]
On the other hand, the representation $\widetilde{\rho}$ can be obtained by lifting the image under $\rho$ of the generators of the knot group $\Gamma$ to $\widetilde{\PSL}_2(\bbR)$. Therefore given $\ell$ as a word in the generators of $\Gamma$, we can compute $\widetilde{\rho}(\ell)$ which is well-defined by Lemma \ref{lem:LiftingLongitudeIsWellDefined}. Comparing $\widetilde{\rho}(\ell)$ and $s(\ell)$, we get $e_\rho([F])$. We demonstrate this computation with an example below.

\begin{example}
\label{exa:TwoBridgeComputation}
    Let $M$ be the complement of the knot $7_3$ in the Rolfsen table. The knot $7_3$ is a two-bridge knot that corresponds to the fraction $13/9$. The knot group $\Gamma$ has the presentation 
    \[
    \langle a,b \mid a w = wb \rangle,
    \]
    where $w = ba^{-1}bab^{-1}aba^{-1}bab^{-1}a$. The longitude that commutes with $a$ is $\ell = wva^{-8}$ where $v$ is $w$ spelled backwards. The discrete faithful representation of $\Gamma$ is given by 
    \[
    a \mapsto \begin{pmatrix}
    1 & 1 \\ 0 & 1
    \end{pmatrix}, \quad 
    b \mapsto \begin{pmatrix}
    1 & 0 \\ z & 1
    \end{pmatrix}
    \]
    where $z$ is a complex root of $\Lambda (z) = z^6-5z^5+9z^4-4z^3-6z^2+5z+1$. The trace field of $\Gamma$ is $\bbQ(z)$ which has two real places and no proper subfield except for $\bbQ$. Let us order the real places of $\bbQ(z)$ by the order of the real roots of $\Lambda$ in $\bbR$. Let $\rho_1$ and $\rho_2$ be the corresponding Galois conjugates of the discrete faithful representation. The lifts $\widetilde{\rho}_i$ can be chosen to be
    \[
        \widetilde{\rho}_i(a) = \left(\frac{1}{5} + \frac{2}{5}i,\arctan\left(\frac{1}{2}\right)\right),
        \hspace{6mm}
        \widetilde{\rho}_i(b) = \left(\frac{zi}{2-zi},-\arctan\left(\frac{z}{2}\right)\right).
    \]
    Comparing $\widetilde{\rho}_i(\ell)$ and $s(\ell)$, we get
    \[
    \widetilde{\rho}_1(\ell) = s(\ell)c^3, \quad \widetilde{\rho}_2(\ell) = s(\ell)c.
    \]
    Therefore, $e_{\rho_1}([F]) = 3$ and $e_{\rho_2}([F]) = 1$. Since the genus of the knot $7_3$ is 2 \cite{knotinfo}, the Thurston norm of $[F]$ is $\Vert[F]\Vert = 3$. The knot $7_3$ satisfies the hypothesis of Theorem \ref{thm:ModifiedCalegari} and thus has no totally geodesic surfaces.
\end{example}

\subsection{Applications}

To prove Theorem \ref{thm:TGSKnotsUnderNineCrossings}, we checked that the knots in \eqref{eq:NoTGSKnots} satisfy the conditions in Theorem \ref{thm:ModifiedCalegari}. We made some preliminary observations to narrow the list of knots to which Theorem \ref{thm:ModifiedCalegari} can be applied. Using SnapPy \cite{SnapPy}, we found $54$ knots, among the $79$ hyperbolic knots with at most $9$ crossings, whose trace field has at least one real place. These 54 knots are:
\begin{equation}
    \label{eq:54Knots}
    \begin{aligned}
          \{
          & 5_2, 6_2, 7_2, 7_3, 7_4, 7_5, 7_6, 8_2, 8_4, 8_5, 8_6, 8_7, 8_{10}, 8_{14}, 8_{15}, 8_{16},
          8_{18}, 8_{20}, \\
          & 9_2, 9_3, 9_4,
          9_5, 9_6, 9_7,
          9_8, 9_9, 9_{10}, 9_{11}, 9_{12}, 9_{13}, 9_{15}, 9_{16}, 9_{17}, 9_{18}, 9_{20}, 9_{21}, \\
          & 9_{22}, 9_{23}, 9_{24}, 9_{25}, 9_{26}, 9_{29}, 9_{31}, 9_{32}, 9_{34}, 9_{35}, 9_{36}, 9_{38},
          9_{39}, 9_{42},
          9_{43}, 9_{45}, 9_{48}, 9_{49}
          \}
    \end{aligned}
\end{equation}

There is exactly one knot listed above whose trace field contains a subfield other than $\bbQ$. The trace field of the knot $8_{18}$ is $\bbQ(\alpha)$ where $\displaystyle \alpha = \frac{1}{2}\left(1-i\sqrt{-5+4\sqrt{2}}\right)$ has minimal polynomial
\(
x^4 -2x^3-x^2+2x-1
\).
The field $\bbQ(\alpha)$ contains $\bbQ(\sqrt{2})$ as a subfield and therefore does not satisfy the condition in Theorem \ref{thm:ModifiedCalegari}. 

Knots with genus one cannot satisfy the inequality $e_\rho([F]) < \Vert[F]\Vert$ on $H_2(M,\partial M;\bbZ)$. Without loss of generality, let $F$ be the genus one Seifert surface of $M$. By the Milnor-Wood inequality, we have
\[
|e_\rho([F])| \leq 1.
\]
Therefore, $|e_\rho([F])| = 1 = \Vert[F]\Vert$. Using the data from \cite{knotinfo}, we can identify genus one knots in \eqref{eq:54Knots}. As a consequence, Theorem \ref{thm:ModifiedCalegari} does not apply to the knots 
\[
\{5_2, 7_2, 7_4, 9_2, 9_5, 9_{35}\} 
\]
in \eqref{eq:54Knots}. The complements of the knots $5_2, 7_2, 7_4, 9_2$ and $9_{35}$ are known to contain a unique totally geodesic surface. 

Finally, \cite[Corollary 4.6]{C06} applies to fibered knot complements. In particular, the conditions in  Theorem \ref{thm:ModifiedCalegari} hold at all real places of the trace field of the knot. Therefore, we can rule out these knots as well. The list of knots for which we need to check the inequality between the Euler class and the Thurston norm is now narrowed down to the following $24$ knots:
\begin{equation}
\label{eq:24Knots}
    \begin{aligned}
          \{
          & 7_3, 7_5, 8_4, 8_6, 8_{14}, 8_{15}, 9_3, 9_4, 9_6, 9_7, 9_8, 9_9,  \\
          & 9_{10}, 9_{12}, 9_{13}, 9_{15}, 9_{16}, 9_{18}, 9_{21}, 9_{23}, 9_{25}, 9_{38}, 9_{39}, 9_{49}\}
    \end{aligned}
\end{equation}
We computed $e_{\rho}([F])$ where $\rho$ is a Galois conjugate of the discrete faithful representation to $\PSL_2(\bbR)$ and $F$ is a Seifert surface in the knot complement using the procedure outlined in Section \ref{subsec:CalculatingEulerClass}. 

We make some comments about the procedure of the computation for the benefit of interested readers. When the knot is a two-bridge knot, the discrete faithful representation of the knot group has a nice formula, as demonstrated in Example \ref{exa:TwoBridgeComputation}. The image of this representation already has entries in the trace field which is convenient for the purpose of computing $\rho$. For knots that are not two-bridge, to obtain $\rho$, we computed a discrete faithful representation using SnapPy \cite{SnapPy}. Note that the image of this representation can lie in a degree two extension of the trace field \cite{MR03}. This happens for the knot $9_{16}$. After a further conjugation, we obtained a discrete faithful representation of the knot group with entries in the trace
field from which $\rho$ can be computed. We lift $\rho$ to obtain $\widetilde{\rho}$ by picking pre-images of the generators of $\text{Im}(\rho)$ so that the group relations are satisfied. Finally to obtain $e_\rho(F)$, we used numerical approximations in order to reduce the computing time. Since the only possible values of $e_\rho(F)$ are $\bbZ$, a discrete set of $\bbR$, the results are reliable.      

The result is reported in Table \ref{tab:EulerClassOfKnots}. In this table, the trace field is viewed as a simple extension of $\bbQ$ by $\alpha$ whose minimal polynomial $m_\alpha$ is recorded in the corresponding column. The values of $e_\rho(F)$ at different real places of the trace field is given as a tuple which is ordered according to the ordering of the real roots of $m_\alpha$ in $\bbR$.  

\begin{proof}[Proof of Theorem \ref{thm:TGSKnotsUnderNineCrossings}]
Among the knots in \eqref{eq:NoTGSKnots}, the following 23 knots
\begin{equation*}
    \begin{aligned}
          \{
             & 6_2, 7_6, 8_2, 8_5, 8_7, 8_{10}, 8_{16}, 8_{20}, \\
             & 9_{11}, 9_{17}, 9_{20}, 9_{22}, 9_{24}, 9_{26}, 9_{29}, 9_{31}, \\
             & 9_{32}, 9_{34}, 9_{36}, 9_{42}, 9_{43}, 9_{45}, 9_{48} 
          \}
    \end{aligned}    
\end{equation*}
are fibered knots \cite{knotinfo}. The complement of these knots does not contain any totally geodesic surface by \cite[Corollary 4.6]{C06}. The remaining $24$ knots satisfy the condition of Theorem \ref{thm:ModifiedCalegari} (see Table \ref{tab:EulerClassOfKnots}) and therefore do not contain any totally geodesic surfaces. 
\end{proof}

\begin{table}[H]
    \centering
    \footnotesize
    \begin{tabular}{|c|c|c|p{0.6\linewidth}|c|}
    \hline
        \bfseries Knot & \bfseries Genus & \bfseries $p/q$ & \bfseries Trace Field & \bfseries $e_\sigma(F)$\\
       \hline \hline
       $7_3$  & 2 & $13/9$ &$ x^6-5 x^5+9 x^4-4 x^3-6 x^2+5 x+1$ & (3,1)\\
       \hline
       $7_5$  & 2 & $17/7$ & $x^8+3 x^7+7 x^6+10 x^5+11 x^4+10 x^3+6 x^2+4 x+1$ & (3,1)\\
       \hline \hline
       $8_4$  & 2 & $19/5$ &$x^9+9 x^8+32 x^7+55 x^6+45 x^5+19 x^4+16 x^3+10 x^2-3 x+1$ & (1)\\
       \hline
       $8_6$  & 2 & $23/13$ &$x^{11}-3 x^{10}+10 x^9-19 x^8+32 x^7-40 x^6+40 x^5-29 x^4+15 x^3-x^2-2 x+1$ & $(-1)$\\
       \hline
       $8_{14}$ & 2 & $31/19$ &$ x^{15}-7 x^{14}+26 x^{13}-63 x^{12}+110 x^{11}-146 x^{10}+156 x^9-143 x^8+118 x^7-86 x^6+52 x^5-26 x^4+12 x^3-4 x^2+1$ & $(-1)$\\
       \hline
       $8_{15}$ & 2 & N/A & $x^7 - x^5 -2x^4+2x^3+x^2-2$ & $(-1)$ \\
       \hline \hline
       $9_{3}$ & 3 & $19/13$ & $x^9-7 x^8+20 x^7-25 x^6+x^5+31 x^4-24 x^3-6 x^2+9 x+1$ & $(5,3,1)$\\
       \hline
       $9_{4}$ & 2 & $21/5$ & $x^{10}+11 x^9+49 x^8+112 x^7+140 x^6+107 x^5+79 x^4+56 x^3+15 x^2+7 x+1$ & $(3,1)$\\
       \hline
       $9_{6}$ & 3 & $27/5$ & $x^{12}+16 x^{11}+108 x^{10}+399 x^9+876 x^8+1161 x^7+891 x^6+317 x^5-43 x^4-71 x^3-13 x^2+6 x+1$ & (5,1) \\
       \hline
       $9_{7}$ & 2 & $29/13$ & $x^{14}+3 x^{13}+13 x^{12}+28 x^{11}+62 x^{10}+97 x^9+137 x^8+152 x^7+142 x^6+106 x^5+62 x^4+32 x^3+12 x^2+5 x+1$ & (3,1)\\
       \hline
       $9_{8}$ & 2 & $31/11$ & $x^{15}+9 x^{14}+38 x^{13}+93 x^{12}+134 x^{11}+86 x^{10}-56 x^9-167 x^8-130 x^7-6 x^6+60 x^5+38 x^4+4 x^3-4 x^2+1$ & $(1)$ \\
       \hline
       $9_{9}$ & 3 & $31/9$ & $x^{15}+13 x^{14}+74 x^{13}+237 x^{12}+450 x^{11}+462 x^{10}+104 x^9-295 x^8-254 x^7+74 x^6+168 x^5+18 x^4-52 x^3-12 x^2+8 x+1$ & $(5,3,1)$\\
       \hline
       $9_{10}$ & 2 & $33/23$ & $x^6-6 x^5+13 x^4-9 x^3-6 x^2+8 x+1$ & $(3,1)$\\
       \hline
       $9_{12}$ & 2 & $35/13$ & $x^{17}+9 x^{16}+40 x^{15}+111 x^{14}+209 x^{13}+271 x^{12}+232 x^{11}+106 x^{10}-9 x^9-37 x^8+4 x^7+38 x^6+30 x^5+2 x^4-8 x^3-4 x^2+x+1 $ & $(1)$ \\
       \hline
       $9_{13}$ & 2 & $37/27$ & $x^{18}-17 x^{17}+129 x^{16}-572 x^{15}+1628 x^{14}-3073 x^{13}+3843 x^{12}-3136 x^{11}+1769 x^{10}-1011 x^9+683 x^8-236 x^7-22 x^5-18 x^4+24 x^3+5 x^2+7 x+1$ & $(3,1)$\\
       \hline
       $9_{15}$ & 2 & $39/23$ & $x^{19}-7 x^{18}+30 x^{17}-91 x^{16}+216 x^{15}-420 x^{14}+688 x^{13}-973 x^{12}+1201 x^{11}-1311 x^{10}+1270 x^9-1093 x^8+836 x^7-560 x^6+328 x^5-162 x^4+65 x^3-19 x^2+2 x+1$ & $(1)$ \\
       \hline
       $9_{16}$ & 3 & N/A & $x^8 - x^7 - 4x^6 + 3x^5 + 5x^4 - x^3 - x^2 - 3x - 1$ & $(1,-5)$\\
       \hline
       $9_{18}$  & 2 & $41/17$ & $x^{20}+7 x^{19}+31 x^{18}+98 x^{17}+245 x^{16}+504 x^{15}+876 x^{14}+1312 x^{13}+1708 x^{12}+1951 x^{11}+1959 x^{10}+1730 x^9+1343 x^8+908 x^7+536 x^6+272 x^5+119 x^4+47 x^3+15 x^2+6 x+1$ & $(3,1)$ \\
       \hline
       $9_{21}$  & 2 & $43/25$ & $x^{21}-7 x^{20}+32 x^{19}-105 x^{18}+275 x^{17}-595 x^{16}+1092 x^{15}-1728 x^{14}+2376 x^{13}-2856 x^{12}+3000 x^{11}-2745 x^{10}+2173 x^9-1465 x^8+828 x^7-380 x^6+139 x^5-45 x^4+16 x^3-9 x^2+3 x+1$ & $(1)$\\
       \hline
       $9_{23}$  & 2 & $45/19$ & $x^5+x^4+4 x^3+2 x^2+4 x+1$ & $(1)$\\
       \hline
        $9_{25}$  & 2 & N/A & $x^{25} - 12x^{24} + 159x^{23} - 1141x^{22} + 7777x^{21} - 31289x^{20} + 117521x^{19} - 195155x^{18} + 629488x^{17} + 450445x^{16} + 778646x^{15} + 29361902x^{14} + 94077682x^{13} + 233504902x^{12} + 985758882x^{11} + 3611707834x^{10} + 8911946417x^9 + 15351955982x^8 + 19428268443x^7 + 18502158935x^6 + 13417536797x^5 + 7302123139x^4 + 2904674429x^3 + 750346185x^2 + 99946478x - 13136479$ & $(-1)$\\
       \hline
       $9_{38}$  & 2 & N/A & $x^{11} - x^{10} + 3x^9 - 6x^8 + 7x^7 - 10x^6 + 11x^5 - 9x^4 + 9x^3 - 3x^2 + 3x - 1$ & $(-1)$ \\
       \hline
        $9_{39}$  & 2 & N/A & $x^{11} - 5x^{10} + 16x^9 - 33x^8 + 53x^7 - 62x^6 + 58x^5 - 38x^4 + 19x^3 - 5x^2 + x - 1$ & $(-1)$ \\
       \hline
       $9_{49}$  & $2$ & N/A & $x^3-x-1$ & $(-1)$\\
       \hline
    \end{tabular}
    \caption{Euler number of Galois conjugate of knots with less than 9 crossings}
    \label{tab:EulerClassOfKnots}
\end{table}

\section{Further questions}
\label{sec:Questions}

A common restriction in the existing techniques for studying totally geodesic surfaces is to require that the trace field arising the Fuchsian subgroups is rational (see Proposition \ref{prop:NoClosedTG} and Theorem \ref{thm:ModifiedCalegari}). To appease this restriction, a sensible step is:

\begin{quest}
Find an obstruction to the existence of totally geodesic surfaces in the Kleinian group whose trace field contains proper subfields other than $\bbQ$.  
\end{quest}

There are concrete examples of an infinite family of hyperbolic knots whose complement contains a totally geodesic surface with trace field properly containing $\bbQ$. For example, consider the balanced pretzel knots with $m$-tangles and $n=2k+1$ half twists in each tangle. As shown in \cite[Example 3.1]{AS05}, the totally geodesic Seifert surface in these examples is an $m$-fold cover of the $(m,m,\infty)$ triangle group. Therefore, the trace field of these Fuchsian subgroups is $\bbQ(\cos(\pi/m))$ by \cite[Proposition 2]{Tak77}. Taking $m$ to be a odd integer greater than $3$, we get infinitely many hyperbolic knots whose complement contains a totally geodesic surface with trace field properly containing $\bbQ$. It is natural to ask

\begin{quest}
Is the totally geodesic Seifert surface the only totally geodesic surfaces in the complement of the $m$-tangle balanced pretzel knot?
\end{quest}

We note that the techniques in our paper only apply to the case $m=3$, that is, for the knots $P(n,n,n)$. Moreover even in this case, we can only prove uniqueness of the totally geodesic Seifert surface for $P(3,3,3)$ and verify the proof for $P(5,5,5)$, $P(7,7,7)$ and $P(11,11,11)$. 

Finally, we would like to remark that the results in this paper still do not give a comprehensive description of totally geodesic surfaces in complements of all knots in $\mathcal{K}$. Using SnapPy \cite{SnapPy}, we found some evidence of a totally geodesic surface in the knot $9_{41}$. In particular, the cusp neighborhood of $9_{41}$ (see Figure \ref{fig:CuspNeighborhood_9_41}) shows six straight lines of horoballs with boundary slopes $1/0$ and one straight line of horoballs with boundary slopes $-6$ with respect to the pink fundamental domain of the cusp.   

\begin{figure}[h]
    \centering
    \includegraphics[scale=0.4]{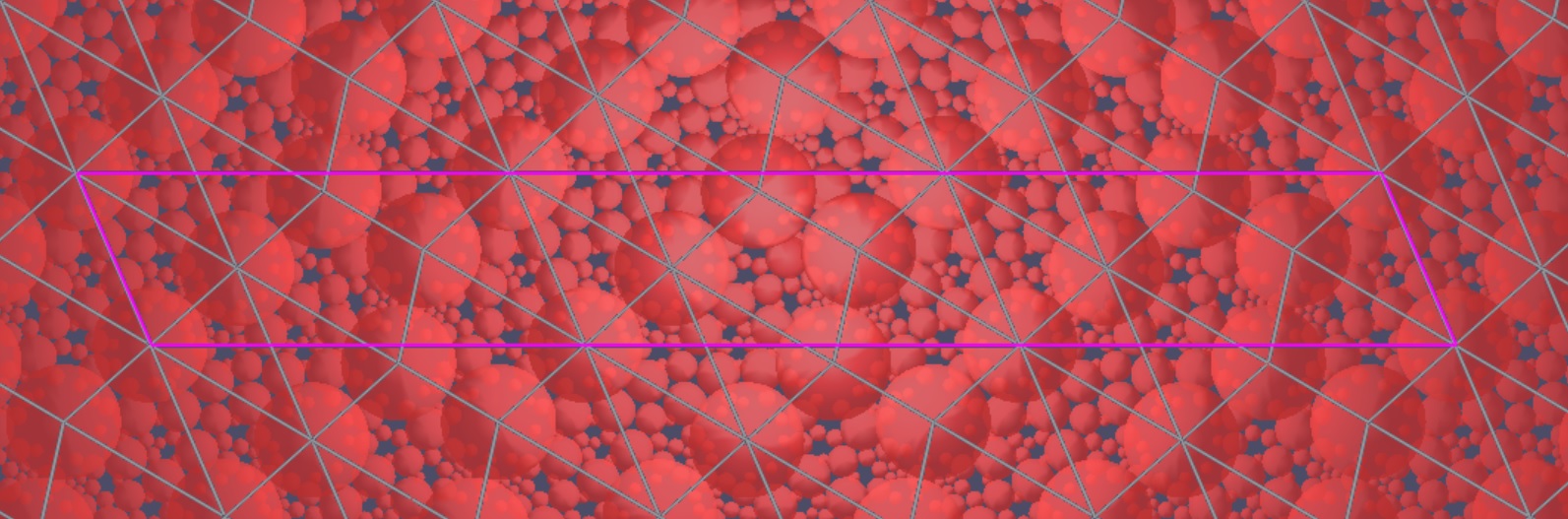}
    \caption{The cusp neighborhood of $9_{41}$}
    \label{fig:CuspNeighborhood_9_41}
\end{figure}

These straight lines of horoballs correspond to lifts of a cusped totally geodesic surface in the complement of $9_{41}$. We expect that these hyperplanes cover a single immersed totally geodesic surfaces with seven cusps, six of which have boundary slope $1/0$ and one of which has boundary slope $-6$. 
\begin{quest}
What is the topology of this surface? Is this a unique totally geodesic surface in the $9_{41}$ complement?
\end{quest}


\end{document}